\documentclass[12pt]{amsart}

\usepackage{amssymb}
\usepackage{latexsym}
\usepackage{verbatim,euscript}
\usepackage{fullpage,color}

\input {cyracc.def}
\numberwithin{equation}{section}
\renewcommand{\theequation}{\thesection$\cdot$\arabic{equation}}

\usepackage[colorlinks=true,linkcolor=blue,urlcolor=my_color,citecolor=magenta]{hyperref}

\definecolor{my_color}{rgb}{0,0.5,0.5}
\definecolor{MIXT}{rgb}{0.4,0.3,0.6}

 \font\tencyr=wncyr10 
\font\tencyi=wncyi10 
\font\tencysc=wncysc10 
\def\rus{\tencyr\cyracc}
\def\rusi{\tencyi\cyracc}
\def\rusc{\tencysc\cyracc}

\newtheorem{thm}{Theorem}[section] 

\newtheorem{utv}[thm]{Claim}
\newtheorem{lm}[thm]{Lemma}
\newtheorem{cor}[thm]{Corollary}
\newtheorem{prop}[thm]{Proposition}

\theoremstyle{remark}
\newtheorem{rmk}[thm]{Remark}

\theoremstyle{definition}
\newtheorem{ex}[thm]{Example}
\newtheorem*{ex-bn}{Example}
\newtheorem{df}{Definition}
\newtheorem{quest}{Question}

\newenvironment{proof*}
{\noindent {\sl Proof.}\quad }{\hfill
$\square$}

\renewcommand{\theequation}{\thesection ${\cdot}$\arabic{equation}}

\newcommand {\ah}{{\mathfrak a}}

\newcommand {\ce}{{\mathfrak c}}

\newcommand {\de}{{\mathfrak d}}

\newcommand {\f}{{\mathfrak f}}
\newcommand {\g}{{\mathfrak g}}
\newcommand {\h}{{\mathfrak h}}
\newcommand {\ka}{{\mathfrak k}}
\newcommand {\el}{{\mathfrak l}}
\newcommand {\m}{{\mathfrak m}}
\newcommand {\n}{{\mathfrak n}}

\newcommand {\q}{{\mathfrak q}}

\newcommand {\es}{{\mathfrak s}}
\newcommand {\te}{{\mathfrak t}}

\newcommand {\z}{{\mathfrak z}}

\newcommand {\sln}{\mathfrak{sl}_n}

\newcommand {\sov}{\mathfrak{so}(V)}

\newcommand {\sone}{\mathfrak{so}_{2n}}
\newcommand {\some}{\mathfrak{so}_{2m}}

\newcommand {\eus}{\EuScript}

\newcommand {\esi}{\varepsilon}
\newcommand {\ap}{\alpha}

\newcommand {\lb}{\lambda}

\newcommand {\ca}{{\mathcal A}}

\newcommand {\cF}{{\mathcal F}}

\newcommand {\co}{{\mathcal O}}

\newcommand {\BZ}{{\mathbb Z}}

\newcommand {\BR}{{\mathbb R}}

\newcommand {\CSS}{{\sf CSS}}

\newcommand {\sgp}{{\sf g.s.}}

\newcommand {\md}{/\!\!/}
\newcommand {\isom}{\stackrel{\sim}{\longrightarrow}}

\newcommand {\ad}{{\mathrm{ad\,}}}

\newcommand {\codim}{{\mathrm{codim\,}}}

\newcommand {\gr}{{\mathrm{gr\,}}}

\newcommand {\Int}{{\mathrm{Int}}}
\newcommand {\Inv}{{\mathsf{Inv}}}
\newcommand {\Lie}{{\mathsf{Lie}}}

\newcommand {\Ima}{{\mathrm{Im\,}}}

\newcommand {\rk}{{\mathrm{rk\,}}}
\newcommand {\spe}{{\mathsf{Spec\,}}}

\newcommand {\trdeg}{{\mathrm{trdeg\,}}}

\newcommand {\tri}{\mathfrak{sl}_2}
\newcommand {\GR}[2]{{\textrm{{\bf #1}}}_{#2}}
\newcommand {\GRt}[2]{{\tilde{\mathrm{{\bf #1}}}}_{#2}}
\newcommand {\ov}{\overline}
\newcommand {\un}{\underline}

\newcommand {\beq}{\begin{equation}}
\newcommand {\eeq}{\end{equation}}

\renewcommand{\le}{\leqslant}
\renewcommand{\ge}{\geqslant}

\newcommand {\bbk}{\mathbb F}

\begin{document}
\setlength{\parskip}{1pt plus 2pt minus 0pt}
\hfill {\scriptsize April 28, 2011}
\vskip1.5ex

\title
{Commuting involutions and degenerations of isotropy representations}
\author[D.\,Panyushev]{Dmitri I. Panyushev}
\address[]{Institute for Information Transmission Problems of the R.A.S., 
 \hfil\break\indent B. Karetnyi per. 19, Moscow 
127994, Russia}
\email{panyushev@iitp.ru}
\urladdr{\url{http://www.mccme.ru/~panyush}}
\begin{abstract}
Let $\sigma_1$ and $\sigma_2$ be commuting involutions of a semisimple algebraic group
$G$.
This yields a $\BZ_2\times \BZ_2$-grading of 
$\g=\Lie(G)$, $\g=\bigoplus_{i,j=0,1}\g_{ij}$, and we study invariant-theoretic 
aspects of this decomposition. Let $\g\langle\sigma_1\rangle$ be the $\BZ_2$-contraction of $\g$ determined by $\sigma_1$. Then both $\sigma_2$ and $\sigma_3:=\sigma_1\sigma_2$
remain involutions of the non-reductive Lie algebra $\g\langle\sigma_1\rangle$.
The  isotropy representations related to $(\g\langle\sigma_1\rangle, \sigma_2)$ and 
$(\g\langle\sigma_1\rangle, \sigma_3)$
are degenerations of the isotropy representations related to $(\g, {\sigma_2})$ and $(\g, {\sigma_3})$, respectively.
We show that these degenerated isotropy representations retain many good properties. For instance, they always have a generic stabiliser and their algebras of invariants
are often polynomial. We also develop some theory  on Cartan subspaces for various
$\BZ_2$-gradings associated with the $\BZ_2\times \BZ_2$-grading of $\g$.
\end{abstract}
\maketitle

\tableofcontents
\section*{Introduction}
\noindent
The ground field $\bbk$ is algebraically closed and $\mathsf{char\,} F=0$. 
Let $\sigma_1$ and $\sigma_2$ be commuting involutions of a connected semisimple 
algebraic group $G$. This yields a $\mathbb Z_2\times \mathbb Z_2$-grading of 
$\g=\Lie(G)$:
\beq   
\label{eq:decomp0}
\g=\bigoplus_{i,j=0,1}\g_{ij}, \ \text{ where }\ \g_{ij}=\{x\in\g\mid   \sigma_1(x)=(-1)^ix \ \ \& \ \ 
\sigma_2(x)=(-1)^jx\}.
\eeq
Since $\sigma_1,\sigma_2$, and $\sigma_3=\sigma_1\sigma_2$ are  pairwise 
commuting involutions, this decomposition possesses an 
$\mathbb S_3$-symmetry, and following \cite{vergne} we say that \eqref{eq:decomp0}
is a {\it quaternionic decomposition}. 
For, if $(\ap,\beta,\gamma)$ is any permutation of the indices $(01,10,11)$, then 
$[\g_{00},\g_\ap] \subset \g_\ap$ and  $[\g_\ap,\g_\beta] \subset \g_\gamma$.
Various problems related to this structure (in the setting of Lie groups or Lie algebras) 
has been studied before. Let $G^{\sigma_i}$ be  
the fixed-point subgroup of $\sigma_i$. 
In  \cite{matsuki}, Matsuki obtained fundamental results on $(G^{\sigma_1}, G^{\sigma_2})$
double cosets in $G$.
The quotient variety $G\md (G^{\sigma_1}{\,\times\,} G^{\sigma_2})$ 
has been studied by Helminck and Schwarz \cite{helm-gerry1, helm-gerry2}. 
In \cite{vergne}, Vergne applies quaternionic decompositions to the study of the 
Kostant--Sekiguchi correspondence. The conjugacy classes of pairs of commuting 
involutions are classified,
see \cite{kollross} and references therein. In \cite{vinb}, commuting conjugate
involutions $\sigma_1,\sigma_2,\sigma_3$ ({\it triads}) are used in the proof of Rosenfeld's 
conjecture on the existence of elliptic planes over the tensor product of two composition 
algebras.

In this article, we consider some  other 
aspects of  quaternionic decomposition~\eqref{eq:decomp0}.
This decomposition embraces several $\BZ_2$-gradings of $\g$ and
the subalgebras $\g^{\sigma_i}=\Lie (G^{\sigma_i})$. 
The involutions $\sigma_1,\sigma_2,\sigma_3$ are said to be {\it big\/} and  the 
induced involutions of $\g^{\sigma_1},\g^{\sigma_2},\g^{\sigma_3}$ are said to be {\it little\/}.
The same terminology applies to the corresponding $\BZ_2$-gradings, symmetric spaces,
and Cartan subspaces (= \CSS\ for short).
Our ultimate goal is to elaborate on invariant-theoretic
properties of degenerations of  isotropy representations involved. To this end, 
we develop some theory on the corresponding \CSS\ and on 
the triples of commuting involutions containing conjugate involutions. 

A {\it dyad} is a pair of different commuting conjugate involutions. 
Any dyad gives rise to a quaternionic decomposition with additional
"symmetry". For a fixed $\sigma_1$, we 
provide a universal construction of dyads $\{\sigma_1,\sigma_2\}$, using automorphisms of 
order 4 (Prop.~\ref{prop:converse}). This implies that 
every involution can be a member of a dyad.
We also describe
all possible conjugacy classes of involutions $\sigma_3=\sigma_1\sigma_2$ via the 
restricted  root system of $G/G^{\sigma_1}$.
\\  \indent
or $\sigma_3$.
Let $\ce_{\ap}$ be a \CSS\ of $\g_{\ap}$ (i.e., $\ce_\ap$ is a little \CSS). 
Then $\ce_{\ap}$ is contained in
a \CSS\ of  $\g_{\ap}\oplus\g_{\gamma}$ or $\g_{\ap}\oplus\g_{\beta}$
(i.e., of the $(-1)$-eigenspaces of  two big involutions).
If  at least one  such embedding
appears to be an equality, then we call it a {\it coincidence of Cartan subspaces}. 
Clearly, there are totally six possibilities for such coincidences.
This  can also be expressed as the equality of ranks for certain little and big symmetric
spaces. 
We prove two sufficient conditions for a coincidence of \CSS:
\\       
{\textbullet} \ If $\sigma_1$ is an involution of maximal rank, then
 $\ce_{10}$ is  also a \CSS\ in $\g_{01}\oplus\g_{10}$ and
$\ce_{11}$ is  also a \CSS\ in $\g_{01}\oplus\g_{11}$ 
(Theorem~\ref{thm:sovpad1}).
\\   
{\textbullet} \ If $\{\sigma_1,\sigma_2\}$ is a dyad, 
then $\ce_{11}$ is a \CSS\ in both
$\g_{01}\oplus\g_{11}$ and $\g_{10}\oplus\g_{11}$ (Theorem~\ref{thm:sovpad2}).

Let $G_{0\star}$ denote the connected subgroup of $G$ with Lie algebra $\g_{0\star}:=
\g_{00}\oplus\g_{01}$,
and $\g_{1\star}:=\g_{10}\oplus\g_{11}$. Then
$(G_{0\star}{:} \g_{1\star})$ is the isotropy representation associated with 
the (big) symmetric space $G/G_{0\star}$. 
We observe that each big symmetric space of $G$ admits two degenerations to
symmetric spaces of non-reductive algebraic groups; accordingly,
each big isotropy representation also admits two degenerations. 
The corresponding non-reductive Lie algebras are $\BZ_2$-contractions of $\g$ in 
the sense of \cite{rims07,coadj07}.
Let $\g\langle\sigma_i\rangle$ be the $\BZ_2$-contraction of $\g$ determined by $\sigma_i$. For instance, $\g\langle\sigma_2\rangle=\g_{\star 0}\ltimes\g_{\star 1}^a$, where
$\g_{\star 0}=\g_{00}\oplus\g_{10}$ and $\g_{\star 1}=\g_{01}\oplus\g_{11}$ are the eigenspaces of $\sigma_2$, and the superscript `$a$' means that $\g_{\star 1}$ is regarded an abelian ideal.  The corresponding connected group is $G\langle\sigma_2\rangle=
G_{\star 0}\ltimes \exp(\g_{\star 1}^a)$.
Remarkably, both $\sigma_1$ and $\sigma_3$ remain involutions of the {\it group\/} 
$G\langle\sigma_2\rangle$ and  {\it algebra\/}  $\g\langle\sigma_2\rangle$. 
Then
$\g\langle\sigma_2\rangle^{\sigma_1}=\g_{00}\ltimes\g_{01}^a=:\ka$, which is the
$\BZ_2$-contraction of $\g^{\sigma_1}=\g_{0\star}$ determined by $\sigma_2$, i.e., 
$\g\langle\sigma_2\rangle^{\sigma_1}=\g^{\sigma_1}\langle\sigma_2\rangle$.
The $(-1)$-eigenspace of $\sigma_1$ 
in $\g\langle\sigma_2\rangle$ is the $\ka$-module denoted by $\g_{10}{\,\propto\,}\g_{11}$.
Likewise, starting with the $\BZ_2$-contraction
$\g\langle\sigma_3\rangle$, we still obtain $\g\langle\sigma_3\rangle^{\sigma_1}=\ka$,  
whereas  the $(-1)$-eigenspace of $\sigma_1$ in $\g\langle\sigma_3\rangle$ is
the $\ka$-module denoted by $\g_{11}{\,\propto\,}\g_{10}$.
\\ \indent
Let  $G_{00}$ be the connected subgroup of $G$ with Lie algebra $\g_{00}$ and
$N_{01}:=\exp(\g_{01}^a)$.  Then 
$K=G_{00}\ltimes N_{01}$ is  the identity component of  both $G\langle\sigma_2\rangle^{\sigma_1}$ and $G\langle\sigma_3\rangle^{\sigma_1}$,
and  $\Lie(K)=\ka$. 
The symmetric spaces $G\langle\sigma_2\rangle /K$ and $G\langle\sigma_3\rangle/K$
can be regarded as degenerations of $G/G_{0\star}$, and
we prove that $G\langle\sigma_2\rangle /K$ and $G\langle\sigma_3\rangle/K$ are affine
varieties.
The $K$-modules  $\g_{10}{\,\propto\,}\g_{11}$ and $\g_{11}{\,\propto\,}\g_{10}$ have the same underlying
vector space $\g_{1\star}$ and afford the same action of the reductive subgroup
$G_{00}\subset K$. But the actions of the  unipotent radical 
$N_{01}\subset K$ are different (see Section~\ref{sect4} for precise formulae).
We show that 
$V:=\g_{10}{\,\propto\,}\g_{11}$ and 
$V^*:=\g_{11}{\,\propto\,}\g_{10}$ are dual $K$-modules. 
Thus, the isotropy representations
$(K:V)$ and $(K:V^*)$ are different degenerations of 
$(G_{0\star}: \g_{1\star})$.
\\ \indent
Altogether, there are six degenerations of three big isotropy representations, and we study
invariant-theoretic properties of these {\it degenerated isotropy representations}.  
This generalises the setting of \cite{coadj07}, where 
the coadjoint representation of $\BZ_2$-contractions is studied.
The main idea behind our considerations is that, although $K$ is non-reductive, 
$(K{:}V)$ retains many good 
invariant-theoretic properties of  $(G_{0\star}: \g_{1\star})$.
Coincidences of \CSS\ fit in this setting as follows.
Suppose that $\ce_{10}$ is also a \CSS\ in $\g_{10}\oplus\g_{11}$. Then
$\bbk[V]^{N_{01}}\simeq \bbk[\g_{10}]$, 
$\bbk[V]^K\simeq \bbk[\g_{10}]^{G_{00}}$, and $\bbk[V]$ is a free 
$\bbk[V]^K$-module (Theorem~\ref{thm:no-extra-inv}). Moreover, 
if  $\sigma_1,\sigma_2$, and $\sigma_3$ 
are conjugate, then all conceivable coincidences of \CSS\ do occur, and 
such a simple description of invariants applies to all six degenerated isotropy representations.

Most of the quaternionic decompositions have at least one coincidence of \CSS.
But this is not always the case, and  examples are given in Section~\ref{sect3}.
In general, we prove that {\sf (i)} the $K$-module $V$ always has a generic stabiliser,
\  {\sf (ii)}  $\trdeg \bbk(V)^K=\trdeg \bbk(\g_{1\star})^{G_{0\star}}$,  and 
\ {\sf (iii)}  $\bbk(V)^K$ is the fraction field of $\bbk[V]^K$
(Theorem~\ref{thm:sgp-i-pgp}). Hence 
this degeneration does not affect the transcendence degree of fields and algebras 
of invariants. Furthermore, the algebra $\bbk[V]^K$ is bi-graded and there is a 
`contraction method' for 
obtaining $K$-invariants on $V$ and $V^*$ from 
$\bbk[\g_{1\star}]^{G_{0\star}}$. Using this method,
we describe $\bbk[V]^K$  under less restrictive assumptions than a coincidence 
of \CSS\ (Theorem~\ref{thm:main5}).

As a by-product of our methods, we prove polynomiality of the algebras of invariant differential operators for many degenerations of big symmetric spaces. We also show that results of \cite{gonz-helg}
on `Invariant differential operators on Grassmann manifolds'
can be better understood in the framework of quaternionic decomposition, see
Section~\ref{sect6}. 

Throughout, $G$ is a connected semisimple  algebraic group and
$\g=\Lie(G)$. 

--  \  $\n_\g(\ah)$ (resp. $\z_\g(\ah)$) is the normaliser (resp. centraliser) 
of a subspace  $\ah\subset\g$.               

-- \ the centraliser in $\g$ of $x\in\g$ is denoted by $\z_\g(x)$ or $\g^x$.

-- \ If $X$ is an irreducible variety, then $\bbk[X]$ is the algebra of regular functions 
and $\bbk(X)$ is the field of rational functions on $X$. If $X$ is acted upon by an algebraic group $\ca$, then $\bbk[X]^\ca$ and $\bbk(X)^\ca$ denote the respective $\ca$-invariant functions.

-- \ If $\bbk[X]^\ca$ is finitely generated, then $X\md \ca:=\spe(\bbk[X]^\ca)$.

\section{Generalities on involutions and isotropy representations}
\label{sect1}

\noindent
Our main object is a connected semisimple algebraic group $G$ with Lie algebra $\g$.
The set of all involutions of $\g$ is denoted by $\mathsf{Inv}(\g)$. 
The group of inner automorphisms  
$\mathsf{Int}(G)\simeq G/Z(G)$ acts on $\mathsf{Inv}(\g)$ by conjugation.
Two involutions are said to be {\it conjugate}, if they lie in the same $\mathsf{Int}(G)$-orbit.
If $\sigma\in\mathsf{Inv}(\g)$, then 
$\g=\g_0\oplus\g_1$ is the corresponding $\mathbb Z_2$-grading of
$\g$. That is, $\g_i=\{x\in\g\mid \sigma(x)=(-1)^ix\}$. We also say that 
$(\g,\g_0)$ is a {\it symmetric pair}.
Whenever we wish to stress that $\g_0$ and $\g_1$ 
are determined by $\sigma$, we write $\g^\sigma$ and $\g_1^{(\sigma)}$ for them.
We assume that $\sigma$ is induced by an involution of $G$, which is denoted by the same letter.
The connected subgroup of $G$ with Lie algebra $\g_0$ is denoted by $G_0$, while the fixed-point subgroup of $\sigma$ is denoted by $G^\sigma$. Hence $G_0$ is the identity 
component of $G^\sigma$, and $G_0=G^\sigma$ if $G$ is simply-connected.
The representation of $G_0$ in $\g_1$, denoted $(G_0:\g_1)$, is  the {\it isotropy representation\/} 
of the symmetric space $G/G_0$.

We freely use invariant-theoretic results on the  $G_0$-action 
on $\g_1$ or  $G/G_0$, as exposed in \cite{kr71},\cite{ri82}.  
A {\it Cartan subspace\/} (=\CSS) is a maximal subspace of $\g_1$ consisting of pairwise 
commuting semisimple elements. All \CSS\ of $\g_1$ are $G_0$-conjugate and their 
common dimension is called the {\it rank\/} of the symmetric space $G/G_0$, denoted
$\rk(G/G_0)$.
 
The Cartan subspaces are characterised by the following property: 
\begin{itemize}
\item[\refstepcounter{equation}(\theequation)\label{char-prop}]   
{\it Suppose that a subspace $\ah\subset \g_1$ consists of pairwise commuting
semisimple elements. Then $\ah$ is a \CSS\ if and only if $\z_\g(\ah)\cap\g_1=\ah$}
\ \cite[Ch.\,I]{kr71}.
\end{itemize}

\noindent
Let $\ce$ be a \CSS\ of $\g_1$. Then every semisimple element of $\g_1$ is 
$G_0$-conjugate to an element of $\ce$ and $G_0{\cdot}\ce$ is dense in $\g_1$.
The {\it generalised Weyl group\/} for $\ce$, $W_0$, is defined to be $N_0(\ce)/Z_0(\ce)$, where
$N_0(\ce)=\{ s\in G_0\mid s{\cdot}\ce\subset \ce\}$ and
$Z_0(\ce)=\{ s\in G_0\mid s{\cdot}x=x \ \forall x\in\ce\}$. 
An element $x\in\g_1$ is called 
$G_0$-{\it regular\/} if the orbit $G_0{\cdot}x$ is of maximal dimension.
Below, we summarise  basic invariant-theoretic properties of the isotropy representation:
\begin{itemize}
\item[\sf --]   Let $x\in\g_1$. 
The orbit $G_0{\cdot}x$ is closed if and only if 
$G_0{\cdot}x\cap\ce\ne\varnothing$;
\item[\sf --]  $x\in\g_1$ is $G_0$-regular and semisimple \ $\Leftrightarrow$ \ $\z_\g(x)\cap\g_1$ is a \CSS;
\item[\sf --]  Each fibre of the quotient morphism $\pi:\g_1\to
\g_1\md G_0=\spe(\bbk[\g_1]^{G_0})$ consists of finitely many $G_0$-orbits. 
The dimension of each fibre equals $\dim\g_1-\dim\ce$.
\item[\sf --]  The restriction of polynomial functions
 $\bbk[\g_1]\to \bbk[\ce]$
induces an isomorphism $\bbk[\g_1]^{G_0}\isom \bbk[\ce]^{W_0}$
(Chevalley's restriction theorem); 
\item[\sf --]  $W_0$ is a  finite reflection group in $GL(\ce)$. Hence $\bbk[\g_1]^{G_0}$ is a 
polynomial algebra and $\g_1\md G_0$ is an affine space of dimension $\dim\ce$; 
\item[\sf --]  $\bbk[\g_1]$ is a free $\bbk[\g_1]^{G_0}$-module.
\end{itemize}
A torus $S$ of $G$ is 
$\sigma$-{\it anisotropic\/}, if $\sigma(s)=s^{-1}$ for all $s\in S$. 
A \CSS\ is  the Lie algebra of a maximal $\sigma$-anisotropic  torus.

\textbullet\quad 
We say that $\sigma\in\mathsf{Inv}(\g)$ is 
{\it  of maximal rank\/} if $\g_1$ contains  a  Cartan subalgebra of $\g$.

\noindent
As is well known, $\dim\g_1-\dim\g_0\le \rk\g$  for any $\sigma$, and the equality holds
if and only if $\sigma$ is of maximal rank.

\textbullet\quad 
Let $\sigma$ be an \un{inner} involution of $\g$. We say that $\sigma$ is 
{\it  quasi-maximal\/} if $\g_1$ contains  a regular semisimple element of $\g$.

\begin{lm}  \label{lm:N=S}
Given $\sigma\in\mathsf{Inv}(\g)$,
the subspace $\g_1$ contains  a regular semisimple element of $\g$
if and only if it contains a regular nilpotent element of $\g$.
\end{lm}
\begin{proof}
For $e\in \g_1$ nilpotent,  there is a normal $\tri$-triple $\{e,h,f\}$, i.e., such that $e,f\in\g_1$ and $h\in\g_0$ \cite[Prop.\,4]{kr71}. If $e$ is regular, then 
$e+f$ is regular semisimple (conjugate to $h$).

Conversely, suppose that $x\in\g_1$ is regular semisimple. 
Then $G_0{\cdot}x$ is a  $G_0$-orbit of maximal dimension and there is also a nilpotent 
element $y\in\g_1$ whose $G_0$-orbit has the same dimension.  Since 
$\dim G{\cdot}z=2\dim G_0{\cdot}z$ for all $z\in\g_1$ \cite[Prop.\,5]{kr71}, we see that $y$ is
regular in $\g$.
\end{proof}
\begin{rmk}  \label{rmk:leva}
It follows from Lemma~\ref{lm:N=S} and \cite[Theorem\,2.3]{theta05} that the quasi-maximal involutions, as well as involutions of maximal rank, form a single $\mathsf{Int}(G)$-orbit.
\end{rmk}

Let $k_0$ (resp. $k_1$) denote  the number of even (resp. odd) exponents of $\g$, so that
$k_0+k_1=\rk\g$.

\begin{prop}    \label{k0-and-k1}
(i)  \ If $\sigma$ is of maximal rank, then $\rk\g_0=k_1$.
\\
(ii) \ If $\sigma$ is quasi-maximal, then $\dim\g_0-\dim\g_1=k_0-k_1$ and 
$\dim(\g_1\md G_0)=k_1$.
\end{prop}\begin{proof}
Part (i) is proved in \cite[Lemma in p.\,1473]{aif99}. \\
In view of Lemma~\ref{lm:N=S}, part (ii)  follows from 
\cite[Theorem\,3.3]{theta05} with $m=2$.
\end{proof}

The quasi-maximal involutions and involutions of maximal rank coincide {\sl if and only if\/} the Weyl group of $\g$ contains $-1$ {\sl if and only if\/} all exponents of $\g$ are odd.
The property of having maximal rank is inheritable in the following sense.

\begin{lm}   \label{lm:inherit}
Let $x\in\g_1$ be semisimple.
If $\sigma$ is of maximal rank, then the restriction of $\sigma$ to  $\z_\g(x)$ and 
$[\z_\g(x),\z_\g(x)]$ is also of maximal rank.
\end{lm}

\noindent
{\bf Warning.} The corresponding assertion for quasi-maximal involutions is not always true.

\section{Quaternionic decompositions, triads,  and  dyads}  
\label{sect2}

\noindent
Let $\sigma_1$ and $\sigma_2$ be different commuting involutions of $\g$. 
The corresponding $\mathbb Z_2\times\mathbb Z_2$-grading of $\g$ is:
\beq  \label{eq:quatern_summa}
\g=\bigoplus_{i,j=0,1}\g_{ij}, \ \text{ where }\ \g_{ij}=\{x\in\g\mid   \sigma_1(x)=(-1)^ix \ \ \& \ \ 
\sigma_2(x)=(-1)^jx\}.
\eeq
We also say that it is a {\it quaternionic decomposition\/} of $\g$ 
(determined by $\sigma_1$ and $\sigma_2$). Let $\sigma_3:=\sigma_1\sigma_2$.
The involutions $\sigma_1,\sigma_2$, and $\sigma_3$ are said to be {\it big}.
The involutions induced on the fixed-point subalgebras of big involutions are said to be {\it
little}. The same terminology applies to the corresponding $\BZ_2$-gradings, \CSS, etc.
Thus, associated with \eqref{eq:quatern_summa}, one has three big 
and three little isotropy representations.
It is convenient for us to organise the summands of  \eqref{eq:quatern_summa}
in a $2\times 2$ ``matrix'':

\hbox to \textwidth{\enspace \refstepcounter{equation} (\theequation) 
\hfil   \label{eq:quatern_matrix}
{\setlength{\unitlength}{0.02in}
\begin{picture}(35,27)(0,5)
\put(-12,7){$\g=$}
    \put(8,15){$\g_{00}$}    \put(28,15){$\g_{01}$}
    \put(8,3){$\g_{10}$}      \put(28,3){$\g_{11}$}
\qbezier[20](5,9),(22,9),(40,9)            
\qbezier[20](23,-1),(23,11),(23,24)      
\put(19.9,7){$\oplus$}   
\put(21,-7){{\color{my_color}$\sigma_2$}}
\put(43,7){{\color{my_color}$\sigma_1$}}
\end{picture}  \hfil
}}
\vskip2.5ex

\noindent
Here the horizontal (resp. vertical) dotted line separates the eigenspaces of  $\sigma_1$ 
(resp. $\sigma_2$), whereas two diagonals of this matrix represent the eigenspaces
of $\sigma_3$.
Hence the first row,  first column, and the main diagonal
represent the three little $\BZ_2$-gradings  (of $\g^{\sigma_1}$,
$\g^{\sigma_2}$, and
$\g^{\sigma_3}$, respectively).

We repeatedly use the following notation for the eigenspaces of $\sigma_1$ and 
$\sigma_2$:

$\g_{0\star}:=\g_{00}\oplus\g_{01}$, \ $\g_{1\star}:=\g_{10}\oplus\g_{11}$, \quad 
$\g_{\star 0}:=\g_{00}\oplus\g_{10}$, \ $\g_{\star 1}:=\g_{01}\oplus\g_{11}$.
\vskip.6ex\noindent
Likewise, the connected subgroup of $G$ corresponding to $\g_{0\star}$ is denoted by $G_{0\star}$, etc.

Following Vinberg \cite[0.3]{vinb}, we say that a triple $\{\sigma_1, \sigma_2, \sigma_3\}
\subset\Inv(\g)$ is a {\it triad\/} if $\sigma_1\sigma_2=\sigma_3$ and all these involutions 
are conjugate. A complete classification of triads is obtained in \cite[Sect.\,3]{vinb}.
Obviously,  triads lead to the most ``symmetric'' quaternionic  decompositions.
Below, we are interested in more general (hence,  less symmetric) decompositions.
\begin{df}
We say that  $\{\sigma_1, \sigma_2\}\subset \Inv(\g)$
is a {\it dyad\/} if $\sigma_1, \sigma_2$
are conjugate and $\sigma_1\sigma_2=\sigma_2\sigma_1$. 
\end{df}

\noindent
That is, we do not require that the third involution 
$\sigma_3=\sigma_1\sigma_2$
is necessarily conjugate to $\sigma_1$. 
Note that the product of two conjugate involutions (not necessarily commuting) is always an
inner automorphism of $\g$. 
For, if $\sigma_2=\Int(g){\cdot}\sigma_1{\cdot}\Int(g^{-1})$, then
$\sigma_1\sigma_2=\Int(\sigma_1(g)g^{-1})$.
Therefore, any triad consists of
inner involutions. A member of a dyad can be an outer involution.
But the third involution, $\sigma_3$, is necessarily inner.
 
We begin with describing the dyads containing involutions of maximal rank.
 
\begin{prop}    \label{prop:ss-inner}
Let $\mu$ be a semisimple inner automorphism of\/ $\g$. Then there exist  involutions 
of maximal rank $\vartheta$ and $\vartheta'$ such that  $\mu=\vartheta\vartheta'$.
\end{prop}
\begin{proof}
By assumption, $\mu=\Int(s)$ for a semisimple element $s\in G$. Choose an involution of
maximal rank $\vartheta$ such that
$s$ belongs to a $\vartheta$-anisotropic maximal torus $T$. 
Clearly, the mapping $T\to T$, $t\mapsto \vartheta(t)t^{-1}=t^{-2}$ is onto.
Therefore, $s=\vartheta(g)g^{-1}$ for some $g\in T$.
Set $\vartheta'=\Int(g){\cdot}\vartheta{\cdot}\Int(g^{-1})$. It is another involution of maximal rank, and a direct computation shows that $\vartheta\vartheta'=\Int(\vartheta(g)g^{-1})=\mu$.
\end{proof}

\begin{prop}   \label{prop:ss-com-inner}
\leavevmode\par
1) \ Suppose that $\mu\in\Inv(\g)$ is  inner. 
Then there are  commuting
involutions of maximal rank, $\vartheta$ and $\vartheta'$,  such that  $\mu=\vartheta\vartheta'$.
Moreover,  $\vartheta$ and $\vartheta'$ induce an involution of maximal rank of
$\g^{\mu}$.

2) \ For commuting involutions of maximal rank, $\vartheta$ and $\vartheta'$, the following conditions are equivalent:
\begin{itemize}
\item[\sf (i)] \ the inner involution $\mu=\vartheta\vartheta'$ is quasi-maximal;
\item[\sf (ii)] \  $\vartheta'\vert_{\g^\vartheta}$ is an involution of maximal rank of 
$\g^{\vartheta}$; 
\item[\sf (iii)] \  $\vartheta\vert_{\g^{\vartheta'}}$ is an involution of maximal rank of 
$\g^{\vartheta'}$. 
\end{itemize}
\end{prop}
\begin{proof}
1) \ Suppose that $\mu=\Int(s)$. Then  $s^2\in Z(G)$ and $\Int(s)=\Int(s^{-1})$.
Choose $\vartheta$ and $\vartheta'$ as in the proof of Proposition~\ref{prop:ss-inner}.
Then $\vartheta\vartheta'=\Int(s)$ and $\vartheta'\vartheta=\Int(s^{-1})$.
Thus, $\vartheta$ and $\vartheta'$ commute.
Consider the quaternionic decomposition of $\g$ determined by $(\vartheta,\vartheta')$:

{\setlength{\unitlength}{0.02in}
\begin{picture}(70,25)(-15,7)
\put(-12,7){$\g=$}
    \put(8,15){$\g_{00}$}    \put(28,15){$\g_{01}$}
    \put(8,3){$\g_{10}$}      \put(28,3){$\g_{11}$}
\qbezier[20](5,9),(22,9),(40,9)              
\qbezier[20](23,-1),(23,11),(23,24)       
\put(20,7){$\oplus$}   
\put(43,7){{\color{my_color}$\vartheta$}}
\put(21,-8){{\color{my_color}$\vartheta'$}}
\end{picture}}  
and  the corresponding `dimension matrix'  \quad $\begin{array}{c|c}
x & y \\  \hline 
u & v  \end{array}$

\vskip2.2ex\noindent
(i.e., $x=\dim\g_{00}$, etc.). 
Let $U$ be a maximal unipotent subgroup of $G$.
Since $\vartheta$ and $\vartheta'$ are of maximal rank,
\[
x+y=x+u=\dim U,\quad  y+v=u+v=\dim U+\rk\g . 
\]
Therefore, $u=y$ and $v-x=\rk\g$.
Consequently, $\dim\g_{11}-\dim\g_{00}=\rk\g=\rk\g^\mu$. Hence  
$\g^\mu=\g_{00}\oplus\g_{11}$ is a $\BZ_2$-grading of maximal rank.

2) \  (i)$\Rightarrow$(ii),(iii).  By the assumption and Proposition~\ref{k0-and-k1}(ii), we have
$x+v-2y=k_0-k_1$.  Consequently, $x=(\dim U-k_1)/2$ and $y=(\dim U+k_1)/2$.
Hence, $y-x=k_1=\rk\g^{\vartheta}=\rk\g^{\vartheta'}$. Therefore, the induced $\BZ_2$-grading of $\g^{\vartheta}$ (or $\g^{\vartheta'}$) is of maximal rank.

Conditions (ii) and (iii) are equivalent, since the dimension matrix is symmetric.

(ii)$\Rightarrow$(i).  By assumption, $\g_{01}$ contains a Cartan subalgebra of 
$\g^{\vartheta}$. Furthermore, the fixed-point subalgebra of any involution 
(e.g. $\g^{\vartheta}$) contains a regular
semisimple element \cite[Lemma~5.3]{ri82}. Hence $\g_{01}$ contains 
a regular semisimple element of $\g$. Thus, $\g_{01}\oplus\g_{10}=
\g_1^{(\mu)}$ contains a regular semisimple element of $\g$, i.e., $\mu$ is quasi-maximal.
\end{proof}

\begin{df}
A triple $\{\vartheta,\vartheta',\mu\}\subset\Inv(\g)$ is said to be {\it canonical\/}, if
$\vartheta$ and $\vartheta'$ are involutions of maximal rank and $\vartheta\vartheta'=\mu$
is quasi-maximal. The corresponding quaternionic decomposition is also called {\it canonical}.
\end{df}

\begin{rmk}    \label{rem:nilp-g1}
By Proposition~\ref{prop:ss-com-inner}, all little involutions involved in 
the canonical decomposition are of maximal rank. Furthermore, $\g_{11}$ contains a Cartan subalgebra of
$\g$. Another interesting feature is that each of three subspaces $\g_{1\star}$, $\g_{\star 1}$,
and $\g_{10}\oplus\g_{01}$ meets all nilpotent $G$-orbits in $\g$. This follows from 
results of Antonyan \cite{leva}.
In Section~\ref{sect5}, we obtain further results on canonical
decompositions.
\end{rmk}
As we shall shortly see, every involution is a member of a dyad.
This raises the natural question: 

{\it What (conjugacy classes of) involutions can occur as products of two members of dyads 
containing given $\sigma\in\Inv(\g)$?}

For instance, Proposition~\ref{prop:ss-com-inner}(1) asserts that {\sl every} inner involution as product from a dyad of involutions of maximal rank. We give below a general 
answer in terms of the reduced root system for $\sigma$.
The following is a recipe for constructing dyads.

\begin{prop}  \label{prop:sabi68}
For any $\sigma_1\in\Inv(\g)$, there is  $\phi\in\Int (G)$ such that 
$\phi^4=\mathsf{id}$, 
$\sigma_2:=\phi\sigma_1\phi^{-1}$ commutes with $\sigma_1$, and
$\sigma_1\sigma_2=\phi^2$.  
In particular, $\{\sigma_1,\sigma_2\}$ is a dyad.
\end{prop}
\begin{proof}
Let $S$ be a one-dimensional $\sigma_1$-anisotropic torus in $G$ and 
$s\in S$ be an element such that $s^4\in Z(G)$. Letting 
$\phi=\Int(s)$, one  has $\phi^4=\mathsf{id}$, $\sigma_1\phi\sigma_1=\phi^{-1}$, 
and the remaining relations are easily verified.
\end{proof}

Remarkably, all the dyads are obtained in this way!

\begin{prop}  \label{prop:converse}
Suppose that $\{\sigma_1,\sigma_2\}$ is a dyad. Then there exists 
$\phi\in\Int(G)$ such that $\phi^4=\mathsf{id}$,
$\sigma_2:=\phi\sigma_1\phi^{-1}$,  and $\sigma_1\sigma_2=\phi^2$.
\end{prop}
\begin{proof}
Suppose that $\sigma_2=\Int(g)\sigma_1\Int(g^{-1})$ for some $g\in G$ and 
hence $\sigma_1\sigma_2=\Int(\sigma_1(g)g^{-1})$.
Set $\tilde g=\sigma_1(g)g^{-1}$.  Since $\sigma_1$ and $\sigma_2$ commute, we 
have $\tilde g^2\in Z(G)$,   $\sigma_1(\tilde g)=\tilde g^{-1}$, and 
$\sigma_2(\tilde g)=\tilde g^{-1}$. By \cite[Prop.\,6.3]{ri82}, the property that $\tilde g$ 
is semisimple and $\tilde g=\sigma_1(g)g^{-1}$ guarantees us that
$\tilde g$ is contained in a $\sigma_1$-anisotropic one-dimensional torus.
It follows that there exists $s\in G$ such that $s^2=\tilde g$ and still $\sigma_1(s)=s^{-1}$.
Then $\phi=\Int(s)$ will do. An easy verification is left to the reader.
\end{proof}

Let $C$ be a maximal $\sigma_1$-anisotropic torus (hence $\Lie(C)$ is a \CSS\ in 
$\g_{1\star}$).
Recall that a {\it restricted root\/} of $C$  is 
any non-trivial weight in the decomposition of $\g$ into the sum 
of weight spaces of $C$. Write $\Phi^{C}$ for the set of all restricted roots.
That is, 
\beq   \label{eq:restr-root}
   \g=\g^C\oplus\bigl(\bigoplus_{\gamma\in\Phi^{C}}\g_\gamma \bigr) .
\eeq
We use the additive notation for the operation in $\mathfrak X(C)$, the character 
group of $C$, and regard $\Phi^{C}$ as a subset of the vector space
$\mathfrak X(C)\otimes_\BZ\BR$.
The set 
$\Phi^{C}$ satisfies the usual axioms of finite root systems \cite{helg}. The notable difference from the structure theory of split semisimple Lie algebras is that the root system 
$\Phi^{C}$ can be non-reduced and that $m_\gamma=\dim \g_\gamma$
($\gamma\in\Phi^C$) can be greater
than $1$.

The (universal) construction of dyads described in the proof of Prop.~\ref{prop:sabi68} means that, for given $\sigma_1$, all possible involutions
$\sigma_3$ are obtained in the following way. Take a one-dimensional torus $S\subset C$,
consider the corresponding $\BZ$-grading of $\g$:
\[
  \g=\bigoplus_{i\in \BZ} \g^S(i) ,
\]  
and then  set 
$\g_0=\oplus_{i\in \BZ} \g^S(2i)$ and $\g_1=\oplus_{i\in \BZ} \g^S(2i+1)$.
Alternatively, this can be expressed as follows.
Choose a linear
form $\ell$ on $\mathfrak X(C)\otimes_\BZ\BR$ that takes integral values on $\Phi^{C}$.
(One should assume that $\ell(\Phi^C)\not\subset 2\BZ$.)
Define $\Phi^{C}_{1}$ to be the set of all restricted roots $\gamma$ such that
$\ell(\gamma)$ is odd. 
Then $\bigoplus_{\gamma\in\Phi^{C}_{1}}\g_\gamma$
is the $(-1)$-eigenspace of $\sigma_3$.

\begin{rmk} If $\Phi^C$ is reduced, then  all possible involution 
$\sigma_3$ are associated  
with  the inner involutions of the semisimple Lie algebra with root system $\Phi^C$.
\end{rmk}

\begin{ex}
If $\sigma_1$ is of maximal rank, then $C$ is a maximal torus of $G$ and 
$\Phi^{C}$ is the usual root system of $\g$.
Here all one-dimensional tori are at our disposal, hence $\sigma_3$ can be any inner 
involution (which is already known from Prop.~\ref{prop:ss-com-inner}).
\end{ex}

\begin{ex}
Let $\sigma_1$ be an involution {\sf E\,IX}, i.e., $\g$ is $\GR{E}{8}$ and
$\g^{\sigma_1}$ is $\GR{E}{7}\times\GR{A}{1}$. Here $\dim C=4$ and $\Phi^C$
is of type $\GR{F}{4}$; for $\gamma\in\Phi^C$, one has
$m_\gamma=\begin{cases}  8,  & \text{if  $\gamma$ short} \\  1,&  
\text{if  $\gamma$ long}
\end{cases}$ \ \ , see \cite[Ch.\,X, Table\,6]{helg}  or  \cite[Table\,9]{VO}. 
The Lie algebra $\GR{F}{4}$  has two (conjugacy classes of) inner involutions, and this 
leads to two possibilities for $\sigma_3$. Using information 
on the multiplicities of restricted roots, one easily computes
$\dim\g^{\sigma_3}$, which allows us to identify $\sigma_3$.
The answer is that involution {\sf F\,I} (resp.  {\sf F\,II}) leads to the involution of $\g$ with
$\dim\g^{\sigma_3}=120$  (resp. $136$), i.e., $\sigma_3$ is either {\sf E\,VIII} 
or {\sf E\,IX}.
\end{ex}

\section{Comparing Cartan subspaces}  
\label{sect3}

\noindent
As is explained above,  the quaternionic decomposition \eqref{eq:quatern_matrix} 
embraces six $\BZ_2$-gradings. 
In this section, we compare \CSS\ for  little and big  $\BZ_2$-gradings. 

For $(ij)\ne (00)$, let $\ce_{ij}$ be a \CSS\ of $\g_{ij}$; that is, 
a \CSS\ related to the little $\BZ_2$-grading 
$\g_{00}\oplus\g_{ij}$. 
There are also \CSS\ for three big involutions:

\centerline{
$\ce_{1\star}\subset \g_{1\star}$, \ 
$\ce_{\star 1}\subset \g_{\star 1}$, \ 
$\ce_{\star,1-\star}\subset \g_{\star,1-\star}:=\g_{01}\oplus\g_{10}$.
}
\vskip.8ex\noindent
Each little \CSS\ can be included in two big \CSS. E.g., because $\g_{10}\subset
\g_{1\star}$ and $\g_{10}\subset\g_{\star,1-\star}$, one can choose Cartan subspaces
$\ce_{1\star}$ and 
$\ce_{\star,1-\star}$ such that 
$\ce_{10}\subset \ce_{1\star}$ and $\ce_{10}\subset \ce_{\star,1-\star}$.
If at least one equality occurs among all such inclusions, then this will be referred to as a {\it
coincidence\/} of \CSS\ (for a given quaternionic decomposition).
We obtain two sufficient conditions for such a coincidence to happen and
provide 
examples  of quaternionic decompositions without coincidences of \CSS.
 
 We begin with a preparatory result.
 Let $\kappa$ denote the Killing form on $\g$.
 Set $\m_{ij}=[\g_{ij},\g_{ij}] \subset \g_{00}$ for $(ij)\ne (00)$.
 Clearly, $\m_{ij}$ is an (algebraic) ideal of the reductive algebraic Lie algebra $\g_{00}$.
 Therefore $\m_{ij}$  is also reductive and  $\kappa\vert_{\m_{ij}}$ is non-degenerate.
 
 \begin{lm}  \label{lm:g_11=0}
 Suppose that\/ $\g_{11}=0$. Then 
 \begin{itemize}
\item[\sf (i)] \ $[\m_{01},\g_{10}]=[\m_{10},\g_{01}]=0$; 
\item[\sf (ii)] \ $\kappa(\m_{01},\m_{10})=0$ and \ 
$\m_{01}\cap\m_{10}=\{0\}$; 
\item[\sf (iii)]  \ $\m_{10}\oplus\g_{10}$ and\/ $\m_{01}\oplus\g_{01}$  are disjoint
ideals of $\g$.
\end{itemize}
\end{lm}\begin{proof}
Everything follows from the Jacobi identity, the relation $[\g_{10},\g_{01}]=0$,
and the fact that $\kappa\vert_{\m_{ij}}$ is non-degenerate.
\end{proof}

\begin{cor}  \label{cor:no-proper}
If\/ $\g_{0\star}=\g^{\sigma_1}$ does not contain  proper ideals of\/ $\g$ and $\g_{11}=0$, then
$\g_{01}=0$ as well. \\
(In particular, this applies  if $\sigma_1$ is of maximal rank.)
\end{cor}\begin{proof}
Since $\sigma_1$ is non-trivial, we have $\g_{10}\ne 0$. As 
$\m_{01}\oplus\g_{01}$ is 
a proper ideal of $\g$ lying in $\g_{0\star}$, it must be zero, i.e., $\g_{01}=0$. 
\end{proof}

Now, we are in a position to prove our first result on the coincidence of \CSS.

\begin{thm}               \label{thm:sovpad1}
Let $\g$ be a semisimple Lie algebra and  $\{\sigma_1,\sigma_2,\sigma_3\}$  a triple of involutions of $\g$ such that $\sigma_1\sigma_2=\sigma_3$.
Suppose that $\sigma_1$ is of maximal rank. Then 
(1) any \CSS\ \ $\ce_{11}\subset \g_{11}$
is also a \CSS\ \ in $\g_{\star 1}$, i.e., for $\sigma_2$; 
(2) any \CSS\ \ $\ce_{10}\subset \g_{10}$
is also a \CSS\ \ in $\g_{10}\oplus\g_{01}$, i.e., for $\sigma_3$.
\end{thm}
\begin{proof}
Obviously, (2) is obtained from (1) if we permute $\sigma_2$ and $\sigma_3$.
Therefore, it suffices to prove the first assertion.

Set $\el=\z_\g(\ce_{11})$ and $\es:=[\el,\el]$.  Then $\el$ is a $(\sigma_1,\sigma_2)$-stable
Levi subalgebra of $\g$. Let $\z$ be the centre of $\el$, so that $\el=\es\oplus\z$.
By construction,  $\ce_{11}\subset \z_{11}$. Furthermore, since $\ce_{11}$ is a \CSS\ of 
$\g_{11}$, we have $\z_\g(\ce_{11})\cap \g_{11}=\ce_{11}$, i.e., 
$\es_{11}=0$ and $\ce_{11}= \z_{11}$.

As $\sigma_1$ is of maximal rank, there exists a Cartan subalgebra $\te\subset\g$ such that
$\ce_{11}\subset \te \subset \g_{1\star}$. Then $[\z,\te]=0$ and
therefore $\z\subset\te\subset\g_{1\star}$.
By Lemma~\ref{lm:inherit}, the restriction of $\sigma_1$ to the semisimple algebra $\es$ 
is still of maximal rank. Applying Corollary~\ref{cor:no-proper} to $\es$ in place of $\g$, we 
conclude that $\es_{01}=0$. Because $\z\subset \g_{1\star}$, this also means that
$\el_{01}=0$. Thus,
\[
   \z_\g(\ce_{11})\cap \g_{\star 1} =\z_\g(\ce_{11})\cap \g_{11}=\ce_{11} ,
\]
which is exactly what we need, in view of~\eqref{char-prop}.
\end{proof}

Alternatively, coincidences of \CSS\ in Theorem~\ref{thm:sovpad1} can be expressed
in terms of the rank of symmetric spaces. 
\begin{cor}
If $\vartheta$ and $\sigma$ are commuting involutions and
$\vartheta$ is of maximal rank, then
\begin{center}
   $\rk(G/G^{\sigma})=\rk(G^{\vartheta\sigma}/G^{\vartheta}\cap G^{\sigma})$ \ and \ 
   $\rk(G/G^{\vartheta\sigma})=\rk(G^{\sigma}/G^{\vartheta}\cap G^{\vartheta\sigma})$ .
\end{center}
\end{cor}

\noindent
The following readily follows from the symmetry between $\sigma_1$ and $\sigma_2$:

\begin{cor} \label{cor:sovpad1}
Suppose that both $\sigma_1$ and $\sigma_2$ are of maximal rank. Then 
\begin{itemize}
\item[\sf (i)] \ any \CSS\ \ $\ce_{11}\subset \g_{11}$
is also a \CSS\ \ in $\g_{1\star}$ or in $\g_{\star 1}$, i.e., for $\sigma_1$ or $\sigma_2$;
\item[\sf (ii)] \ any \CSS\ \ $\ce_{01}\subset \g_{01}$
is also a \CSS\ \ in $\g_{10}\oplus\g_{01}$, i.e., for $\sigma_3$;
\item[\sf (iii)] \ any \CSS\ \ $\ce_{10}\subset \g_{10}$
is also a \CSS\ \ in $\g_{10}\oplus\g_{01}$, i.e., for $\sigma_3$.
\end{itemize}
\end{cor}

\begin{rmk}
If both $\sigma_1$ and $\sigma_2$ are of maximal rank, then 
$\ce_{11}$ is a Cartan subalgebra of $\g$ (use Prop.~\ref{prop:ss-com-inner}(1)).
This  provides another explanation for part (i) in Corollary~\ref{cor:sovpad1}.
\end{rmk}

Our second result on the coincidence of \CSS\ concerns arbitrary dyads.

\begin{thm}    \label{thm:sovpad2}
Let $\g$ be a semisimple Lie algebra and  $\{\sigma_1,\sigma_2,\sigma_3\}$  a triple of involutions of $\g$ such that $\sigma_1\sigma_2=\sigma_3$.
Suppose that $\{\sigma_1,\sigma_2\}$ is a dyad. 
Then any \CSS\ \ $\ce_{11}\subset \g_{11}$
is also a \CSS\ \ in $\g_{1\star}$ or $\g_{\star 1}$, i.e., for $\sigma_1$ or $\sigma_2$;
\end{thm}
\begin{proof}
As in the proof of Theorem~\ref{thm:sovpad1}, we consider the $(\sigma_1,\sigma_2)$-stable
Levi subalgebra
\[
         \z_\g(\ce_{11})=:\el=\es\oplus\z ,
\]
where $\es$ is semisimple and $\z$ is the centre of $\el$.
Since $\ce_{11}$ is a \CSS\ in $\g_{11}$, we have
$\el_{11}=\ce_{11}(=\z_{11})$, so that $\es_{11}=\{0\}$, 
and our task is to prove that $\el_{01}=\el_{10}=\{0\}$.

Since $\sigma_1$ and $\sigma_2$ are conjugate, 
$\sigma_2=\Int(g)\sigma_1\Int(g^{-1})$ for some $g\in G$.
Then $\sigma_3=\Int(\tilde g)$, where $\tilde g=\sigma_1(g)g^{-1}$.
Since  $\sigma_1\sigma_2=\sigma_2\sigma_1$ and $\Int(\tilde g) =\Int(\tilde g^{-1})$, 
we also have $\tilde g=g^{-1}\sigma_2(g)$.
The choice of $g$ is not unique, and we are going to demonstrate that
$g$ can be chosen to have some extra properties.

Let $x\in\ce_{11}$ be  generic. In particular,  $x\in\g^{\sigma_3}$, i.e., 
$\Int(\tilde g){\cdot}x=x$.
Making use of two expressions for $\tilde g$, one easily computes that
  
$\quad\sigma_1(\Int(g^{-1}){\cdot} x) =-\Int(g^{-1}){\cdot} x$ and 
$\quad\sigma_2(\Int(g){\cdot} x) =-\Int(g){\cdot}x$. 

\noindent
In particular, 
\beq   \label{eq:dva-srazu}
\mbox{
$\Int(g^{-1}){\cdot}x  \subset \g_{1\star}=\g_1^{(\sigma_1)}$ \  and \ 
$\Int(g){\cdot}x  \subset \g_{\star 1}=\g_1^{(\sigma_2)}$ .}
\eeq
Recently, M.~Bulois \cite[Prop.\,6.6]{bulois}
noticed that any $\BZ_2$-grading $\g=\g_0\oplus\g_1$
has the following  property:

\vskip1ex
\centerline{\it If $x\in\g_1$ is semisimple, then
$G{\cdot}x\cap \g_1=G_0{\cdot}x$.}

\noindent
Applying this, say,  to the first inclusion in \eqref{eq:dva-srazu} shows that 
the identity component of $G^{\sigma_1}$ contains
$p$  such that 
$\Int(g^{-1}){\cdot}x=\Int(p){\cdot}x$, i.e., $gp\in Z_G(x)$, the centraliser of $x$ in $G$.

Replacing $g$ with $gp$ does not affect $\tilde g$ and $\sigma_2$. Therefore, we may
assume that our initial $g$ lies in $Z_G(x)$. Note that this group is always connected.
So far, we did not use the assumption that $x$ is generic in $\ce_{11}$.
For generic $x$, we have $\z_\g(x)=\z_\g(\ce_{11})$. Hence 
$\Lie(Z_G(x))=\el$, and we will write $L$ in place of $Z_G(x)$.
Since $g\in L$, the involutions $\sigma_1\vert_\es$ and
$\sigma_2\vert_\es$ are conjugate with respect to $\Int(S)$, where $S:=(L,L)$ is 
connected and semisimple. Note that $\Lie(S)=\es$.

Let us prove that $\es_{01}=\es_{10}=\{0\}$.  Since $\es_{11}=\{0\}$, it follows
from Lemma~\ref{lm:g_11=0}(iii) that 
$\q_{10}=\es_{10}\oplus [\es_{10},\es_{10}]$ and 
$\q_{01}=\es_{01}\oplus [\es_{01},\es_{01}]$  are disjoint ideals of $\es$.
Thus, $\es^{\sigma_1}\supset \q_{01}$,  $\es^{\sigma_2}\supset \q_{10}$,
and $\es^{\sigma_1}$ and $\es^{\sigma_1}$ are conjugate with respect to $S$.
It is only possible if $\q_{10}=\q_{01}=\{0\}$, i.e., $\es_{10}=\es_{01}=\{0\}$.
This completes the first part of our programme. 

The second part deals with the centre of $\el$. 
Let $\te$ be a Cartan subalgebra of $\g^{\sigma_3}=\g_{00}\oplus\g_{11}$ that 
contains $\ce_{11}$.
Since $\sigma_3$ is inner, $\te$ is actually a Cartan subalgebra of $\g$.
As $\te\subset \z_\g(x)$, $\te$ must contain $\z$. Therefore, $\z\subset \g^{\sigma_3}$
and $\z_{01}=\z_{10}=\{0\}$.

We have proved that $\el_{01}=\el_{10}=\{0\}$. Hence
$
\z_\g(\ce_{11})\cap\g_{1\star}=\z_\g(\ce_{11})\cap\g_{\star 1}=\ce_{11}
$,
and we are done.
\end{proof}

\begin{cor}  \label{cor:triada}
If $\{\sigma_1,\sigma_2,\sigma_3\}$ is a triad, then
every little \CSS\ is also a big \CSS\ in two possible ways. (That is, six coincidences of\/ 
\CSS\ occur.)
\end{cor}

\noindent
Again, the coincidence of \CSS\ can be expressed in terms of ranks of symmetric spaces:
\begin{cor}
If\/ $\{\sigma_1,\sigma_2\}$ is a dyad, then  
$\rk(G^{\sigma_1\sigma_2}/G^{\sigma_1}\cap G^{\sigma_2})= \rk (G/G^{\sigma_1})$.
\end{cor}

Theorems~\ref{thm:sovpad1} and \ref{thm:sovpad2} show that, for many  
triples of commuting involutions,
there is a coincidence of \CSS. However, this is not always the case, and
we provide below two examples of commuting triples without coincidences of \CSS. 

\begin{ex}   \label{ex:h+h}
Let $\g$ be a semisimple Lie algebra and $\sigma$ an involution of
$\g$ with the corresponding $\BZ_2$-grading  $\g=\g_0\oplus\g_1$.
We specify the requirements on $\sigma$ below.
Set $\tilde\g=\g\oplus\g$  and define three involutions of $\tilde\g$ as follows:

$\sigma_1(x_1,x_2)=(x_2,x_1)$, \quad
$\sigma_2(x_1,x_2)=(\sigma(x_1),\sigma(x_2))$, \quad
$\sigma_3=\sigma_1\sigma_2$.

\noindent
Then $\tilde\g^{\sigma_1}=\Delta(\g)$, the diagonal; \ 
$\tilde\g^{\sigma_2}=\g_0\oplus\g_0$; \ 
$\tilde\g^{\sigma_3}=\{ (x,\sigma(x))\mid x\in\g \}$. 
\\
Set $\Delta_-(M):=\{(m,-m)\mid m\in M\}$ for any subspace $M\subset \g$.
Then the  corresponding quaternionic decomposition is: 
\begin{center}
\setlength{\unitlength}{0.02in}
\begin{picture}(90,28)(-10,5)
\put(-30,7){$\tilde\g=$}
    \put(0,15){$\Delta(\g_0)$}    \put(40,15){$\Delta(\g_1)$}
    \put(-2,0){$\Delta_-(\g_0)$}      \put(38,0){$\Delta_-(\g_1)$}
\qbezier[40](-10,9),(32,9),(70,9)            
\qbezier[20](33,-3),(33,11),(33,24)      
\put(29.9,7){$\oplus$}   
\put(31,-12){{\color{my_color}$\sigma_2$}}
\put(80,7){{\color{my_color}$\sigma_1$}}
\end{picture}  
\end{center}
\vskip3ex
Let $\ce$ be a \CSS\ of $\g_1$ and $\dim\ce=r$. Let $\te$ be a Cartan subalgebra of
$\g$ and $\te_0$ a Cartan subalgebra of $\g_0$. 
Both little and big \CSS\ can explicitly be described. We have

\centerline{
$\ce_{01}=\Delta(\ce)\subset \Delta(\g_1)$, \ $\ce_{10}=\Delta_-(\te_0)
\subset\Delta_-(\g_0)$, \ 
$\ce_{11}=\Delta_-(\ce)\subset \Delta_-(\g_1)$
}

and 

\centerline{$\ce_{1\star}=\Delta_-(\te)\subset \Delta_-(\g)$, \ 
$\ce_{\star 1}=\ce\oplus\ce\subset \g_1\oplus\g_1$, \ 
$\ce_{\star,1-\star}=\{t, -\sigma(t)\mid t\in \te\} \subset\{x, -\sigma(x)\mid x\in \g\}$.
}

\noindent
The absence of coincidences means that each big \CSS\ has strictly bigger
dimension than either of two little \CSS\ that can belong to it.
This yields the following conditions:
\begin{gather*}
r=\dim \ce_{11} < \dim \ce_{1\star}=\rk\g, 
\quad \rk\g_0=\dim \ce_{10} < \dim \ce_{1\star}=\rk\g,      \\
r=\dim \ce_{11} < \dim \ce_{\star 1}=2r, 
\quad r=\dim \ce_{01} < \dim \ce_{\star 1}=2r, \\
\rk\g_0=\dim \ce_{10} < \dim \ce_{\star,1-\star}=\rk\g, 
\quad r=\dim \ce_{01} < \dim \ce_{\star,1-\star}=\rk\g . 
\end{gather*}

\noindent
All this amounts to the inequalities
$\rk\g_0<\rk\g$ and $r < \rk\g$. Both are satisfied if and only if
$\sigma$ is outer and not of maximal rank. E.g., such an involution $\sigma$ exists
if $\g$ is a simple Lie algebra of type $\GR{A}{2n+1}$ or $\GR{D}{2n+1}$ or $\GR{E}{6}$.
\end{ex}

A  drawback of Example~\ref{ex:h+h} is that $\tilde\g$ is not simple.
Actually, it is not easy to discover such an example for simple Lie algebras.
A coincidence of some little and big \CSS\ is rather a rule, than exception.
For instance, any triple of commuting  involutions of exceptional Lie algebras
contains either a dyad or an involution of maximal rank, see \cite[Table\,1]{kollross}.
We can also prove that there is always a coincidence of \CSS\  
for the triples of commuting  involutions of $\g=\sln$.
The following example concerns algebras of type $\GR{D}{N}$.

\begin{ex}   \label{ex:so_2n}
$\g=\mathfrak{so}_{2N}$ and $G=SO_{2N}$. Consider three involutions of $\g$ with the following fixed-point
subalgebras: 

$\sigma_1, \sigma_2$:  \   $\mathfrak{gl}_N$; \quad
$\sigma_3$:  \ $\sone\oplus\some$, \ $m+n=N$.

For $N$ even, $\Int(G)$ contains two conjugacy classes of involutions with the
fixed-point subalgebra $\mathfrak{gl}_N$, and we will arrange that $\sigma_1, \sigma_2$
belong to different conjugacy classes. Namely, assume that $\g$ is represented by 
the matrices of order $2N$ that are skew-symmetric w.r.t. the antidiagonal
and define the $\sigma_i$'s using certain diagonal matrices, as follows:

$\sigma_1=\Int\bigl(\text{diag}(\underbrace{i,\dots,i}_{N},\underbrace{-i,\dots,-i}_{N})\bigr)$;

$\sigma_2=\Int\bigl(\text{diag}(\underbrace{i,\dots,i}_{m},\underbrace{-i,\dots,-i}_{n},
\underbrace{i,\dots,i}_{n},\underbrace{-i,\dots,-i}_{m}
)\bigr)$;

$\sigma_3=\sigma_1\sigma_2=\Int\bigl(\text{diag}(
\underbrace{-1,\dots,-1}_{m},\underbrace{1,\dots,1}_{n},
\underbrace{1,\dots,1}_{n},\underbrace{-1,\dots,-1}_{m}
)\bigr)$.

\noindent
Here $i=\sqrt{-1}$. One easily presents
this quaternionic decomposition in the matrix form.  Taking the eigenspaces of
$\sigma_1$ and $\sigma_2$ partition the  matrices in  $\g$
in $16$ blocks, and we indicate the subspace $\g_{ij}$ which each block belongs to:
 
\begin{center} $\g=\begin{pmatrix}
\g_{00} & \g_{01} & \g_{10} & \g_{11} \\
\g_{01} & \g_{00} & \g_{11} & \fbox{$\g_{10}$} \\
\g_{10} & \g_{11} & \g_{00} & \g_{01} \\
\g_{11} & \g_{10} & \g_{01} & \g_{00} 
\end{pmatrix}$  \end{center}
The diagonal blocks consists of  square matrices of order $m,n,n,m$, respectively. 
Hence the framed block 
consists of rectangular matrices of shape $n\times m$.
Then one computes that
\\
$\dim\ce_{01}=\dim\ce_{10}=\min\{n,m\}$, 
$\dim\ce_{11}=[n/2]+[m/2]$; 
\quad and 
\\
$\dim\ce_{1\star}=\dim\ce_{\star 1}=[n+m/2]$, \ 
$\dim\ce_{\star,1-\star}=2\min\{n,m\}$.

If $n,m$ are odd, then $\sigma_1, \sigma_2$ are not conjugate.
Furthermore, if  $n,m$ are odd and $n\ne m$, then there is no coincidence of \CSS,
since $[n/2]+[m/2]<[n+m/2]$ and $\min\{n,m\}< [n+m/2]$.
\end{ex}

\begin{rmk}  \label{rmk:geom-smysl}
A coincidence of Cartan subspaces has the following invariant-theoretic meaning.
Suppose that $\ce_{10}\subset \g_{10}$ is also a \CSS\ in $\g_{10}\oplus\g_{11}=\g_{1\star}$.
Then, in view of Chevalley's restriction theorem, the natural restriction homomorphism
\[
   \varrho: \bbk[\g_{1\star}]^{G_{0\star}} \to \bbk[\g_{10}]^{G_{00}}
\]
is injective, and it makes $\bbk[\g_{10}]^{G_{00}}$ a finite $\bbk[\g_{1\star}]^{G_{0\star}}$-module. The  degree of the finite ring extension equals the ratio
of orders of the corresponding generalised Weyl groups.
It should be noted that $\bbk[\g_{10}]^{G_{00}}$ is always a finite $\varrho(\bbk[\g_{1\star}]^{G_{0\star}})$-module. That is, the point is that such a coincidence of
\CSS\ implies the injectivity of $\varrho$.

Putting this in the geometric form, we obtain the commutative diagram
\begin{center}
$\begin{array}{ccc}   \g_{10} & \stackrel{i}{\hookrightarrow}  &\g_{1\star}  \\
                   \downarrow  & & \downarrow \\
         \g_{10}\md G_{00} &  \stackrel{\bar i}{\rightarrow } & \g_{1\star}\md G_{0\star} 
\end{array}$
\end{center}
where the vertical arrows are quotient morphisms, $i$ is the embedding, 
and the morphism $\bar i$ is finite and surjective.
\end{rmk}

\section{Degenerations of symmetric spaces and isotropy representations}  
\label{sect4}

\noindent
For any $\sigma\in\Inv(\g)$ with $\BZ_2$-grading $\g=\g_0\oplus\g_1$,  
there is a non-reductive Lie algebra, which is a contraction of $\g$. Namely,  the 
semi-direct product $\g\langle\sigma\rangle=\g_0\ltimes\g_1^a$
is called a $\BZ_2$-{\it contraction} of $\g$ (with respect to $\sigma$).
Here the superscript `$a$'  means that  the $\g_0$-module $\g_1$ is regarded as an abelian 
ideal of $\g\langle\sigma\rangle$.  Then 
$\g_1^a$ is also the nilpotent radical of $\g\langle\sigma\rangle$.
The corresponding connected group $G\langle\sigma\rangle$ is a semi-direct product of 
$G_0$ and the abelian unipotent radical $\exp(\g_1^a)$, i.e., 
$G\langle\sigma\rangle=G_0\ltimes \exp(\g_1^a)$.   

Invariant-theoretic properties of the adjoint and coadjoint representations
of $G\langle\sigma\rangle$ have been studied in \cite{rims07, coadj07}. 
By \cite[Thm.\,6.2]{rims07},  the algebra 
$\bbk[\g\langle\sigma\rangle]^{G\langle\sigma\rangle}$ is always polynomial.
In \cite{coadj07}, it is proved  that  the 
algebra  $\bbk[\g\langle\sigma\rangle^*]^{G\langle\sigma\rangle}$
is polynomial in many cases. 
There is also a useful method of ``contraction'' of $G$-invariants.
Namely, to any homogeneous $f\in \bbk[\g]^G$ one can associate an element
of either $\bbk[\g\langle\sigma\rangle]^{G\langle\sigma\rangle}$ or  
$\bbk[\g\langle\sigma\rangle^*]^{G\langle\sigma\rangle}$ \cite[Prop.\,3.1]{coadj07}.
In the context of quaternionic decompositions, we may extend the scope of this method
beyond the (co)adjoint representations.

Recall that we work with a  quaternionic decomposition
\begin{center}
{\setlength{\unitlength}{0.02in}
\begin{picture}(35,22)(-40,5)
\put(-12,7){$\g=$}
    \put(8,15){$\g_{00}$}    \put(28,15){$\g_{01}$}
    \put(8,3){$\g_{10}$}      \put(28,3){$\g_{11}$}
\qbezier[22](5,9),(22,9),(40,9)            
\qbezier[22](23,-1),(23,11),(23,24)      
\put(19.9,7){$\oplus$}   
\put(21,-7){{\color{my_color}$\sigma_2$}}
\put(43,7){{\color{my_color}$\sigma_1$}}
\end{picture}  \hfil
}
\vskip2.5ex
\end{center}
and $\sigma_3=\sigma_1\sigma_2$.
Consider the $\BZ_2$-contraction of $\g^{\sigma_1}$ with respect to $\sigma_2$
(or $\sigma_3$, which is the same), that is,  set $\ka_{01}:=\g^{\sigma_1}\langle\sigma_2\rangle
=\g_{00}\ltimes\g_{01}^a$.
The vector space $\g_{1\star}$ can be regarded as $\ka_{01}$-module in two different ways.
In both cases, the subalgebra $\g_{00}\subset \ka_{01}$ acts on $\g_{1\star}$ as it was in $\g$.
Let $x\in\g_{01}$ and $(y_0,y_1)\in\g_{10}\oplus\g_{11}$. 
For  the action of the abelian nilpotent radical $\g_{01}^a\subset\ka_{01}$, the two
possibilities are:

a) \  $x{\cdot}(y_0,y_1)=(0, [x,y_0])$;

b) \  $x{\cdot}(y_0,y_1)=([x,y_1],0)$.

\noindent
The $\ka_{01}$-modules obtained in this way are denoted by
$\g_{10}{\,\propto\,}\g_{11}$ and $\g_{11}{\,\propto\,}\g_{10}$, respectively.
From a slightly different angle, these possibilities can be realised as follows:

\noindent
a')  \ Take  the $\BZ_2$-contraction of $\g$ with respect to $\sigma_2$, i.e.,
$\g\langle\sigma_2\rangle=\g_{\star 0}\ltimes\g_{\star 1}^a$. 
Identifying $\g$ and $\g\langle\sigma_2\rangle$ as vector
spaces, we notice that 
the linear operators $\sigma_1$ and $\sigma_3$  remain  involutions of the {\sl Lie algebra\/} $\g\langle\sigma_2\rangle$.
Taking the eigenspaces of $\sigma_1$ in $\g\langle\sigma_2\rangle$ yields
$\g\langle\sigma_2\rangle^{\sigma_1}=\ka_{01}$ and
$\g\langle\sigma_2\rangle^{(\sigma_1)}_1=\g_{10}{\,\propto\,}\g_{11}$.

\noindent
b') \ Likewise, starting with 
$\g\langle\sigma_3\rangle$,
we end  up with $\g\langle\sigma_3\rangle^{\sigma_1}=\ka_{01}$ and
$\g\langle\sigma_3\rangle^{(\sigma_1)}_1=\g_{11}{\,\propto\,}\g_{10}$.

We have $G\langle\sigma_2\rangle=G_{\star 0}\ltimes \exp(\g_{\star 1}^a)$ and
both $G_{\star 0}$ and $\g_{\star 1}$ are $\sigma_1$-stable. It follows that 
$\sigma_1$ can be lifted to an involution of $G\langle\sigma_2\rangle$
and the identity
component of $G\langle\sigma_2\rangle^{\sigma_1}$ is
$K_{01}=G_{00}\ltimes \exp(\g_{01}^a)=:G_{00}\ltimes N_{01}$.
(Likewise, $K_{01}$ is the identity
component of $G\langle\sigma_3\rangle^{\sigma_1}$.)
The exponential map $\exp: \g_{01}\to N_{01}$ is an isomorphism of varieties and
the action of $N_{01}$  is given by
\begin{gather}   \label{eq:N01:V}
   \exp(x){\cdot}(y_0,y_1)=(y_0,y_1)+x{\cdot}(y_0,y_1)=(y_0, y_1+[x,y_0]) \  \text{ for } \ (y_0,y_1)\in\g_{10}{\,\propto\,}\g_{11},
\\
   \notag \hphantom{(4.2)}
    \exp(x){\cdot}(y_1,y_0)=(y_1,y_0)+x{\cdot}(y_1,y_0)=(y_1, y_0+[x,y_1])  \ \text{ for } \ (y_1,y_0)\in\g_{11}{\,\propto\,}\g_{10} .
\end{gather}

\noindent
Note that $G_{00}$ is a subgroup of both $K_{01}$ and $G_{0\star}$, and the action of 
$G_{00}$ does not vary under the passage from the $G_{0\star}$-module $\g_{1\star}$ to 
the $K_{01}$-modules $\g_{10}{\,\propto\,}\g_{11}$ or $\g_{11}{\,\propto\,}\g_{10}$.
\\
Summarising the previous discussion, we get the following:

\begin{utv}   \label{utv:claim}
1)  The involution $\sigma_1\in\Inv(\g)$ can be regarded as involution of $\g\langle\sigma_2\rangle$ and of $G\langle\sigma_2\rangle$. The group $K_{01}$ is the identity 
component of $G\langle\sigma_2\rangle^{\sigma_1}$ and the corresponding isotropy 
representation is $(K_{01}: \g_{10}{\,\propto\,}\g_{11})$.
\\ \indent 
2) Likewise, $K_{01}$ is the identity 
component of $G\langle\sigma_3\rangle^{\sigma_1}$ and the corresponding isotropy 
representation is $(K_{01}: \g_{11}{\,\propto\,}\g_{10})$.
\\ \indent
3) Both representations can be understood as different degenerations of  
$(G_{0\star}\,{:}\,\g_{1\star})$. The symmetric spaces 
$G\langle\sigma_2\rangle/K_{01}$ and $G\langle\sigma_3\rangle/K_{01}$ can be regarded
as different degenerations of $G/G_{\star 0}$.
\end{utv}

\begin{lm}
The $K_{01}$-modules $\g_{10}{\,\propto\,}\g_{11}$ and $\g_{11}{\,\propto\,}\g_{10}$ are 
dual to each other.
\end{lm}
\begin{proof}
If  $\bar y=(y_0,y_1)\in\g_{10}{\,\propto\,}\g_{11}$  and 
$\bar z=(z_1,z_0)\in\g_{11}{\,\propto\,}\g_{10}$, then  
$(\bar y,\bar z)\mapsto \kappa(y_0,z_0)+\kappa(y_1,z_1)$ is the required 
$K_{01}$-invariant  pairing.
\end{proof}

\begin{lm}
The homogeneous space $G\langle\sigma_2\rangle/K_{01}$ is affine. 
\end{lm}
\begin{proof} Let $\ca^u$ denote the unipotent radical of an algebraic group $\ca$.
By construction,  $G\langle\sigma_2\rangle^u=\exp(\g_{\star 1}^a)$ and 
$(K_{01})^u=\exp(\g_{01}^a)$. Hence $(K_{01})^u\subset G\langle\sigma_2\rangle^u$.
By \cite[Cor.\,2]{bb63}, this condition is sufficient for 
$G\langle\sigma_2\rangle/K_{01}$ to be affine.
\end{proof}

{\bf Hint}. Using the  notation $\g_{10}{\,\propto\,}\g_{11}$ or $\g_{11}{\,\propto\,}\g_{10}$
always means that the vector space $\g_{1\star}$ is regarded as  a 
$K_{01}$-module in the prescribed way.

Similar notation is used for the other  $\BZ_2$-contractions of $\g$ and their involutions.
For instance, 
$\g\langle\sigma_1\rangle^{\sigma_2}=\g\langle\sigma_3\rangle^{\sigma_2}=\g_{00}\ltimes\g_{10}^a=:\ka_{10}$,
$\g\langle\sigma_1\rangle^{(\sigma_2)}_1=\g_{01}{\,\propto\,}\g_{11}$, and $\g\langle\sigma_3\rangle^{(\sigma_2)}_1=\g_{11}{\,\propto\,}\g_{01}$.
Then $N_{10}=\exp(\g_{10}^a)$ and  $K_{10}=G_{00}\ltimes N_{10}$ acts on $\g_{01}{\,\propto\,}\g_{11}$ or
$\g_{11}{\,\propto\,}\g_{01}$.
In this way, we obtain 6 new isotropy representations related to 3 possible $\BZ_2$-contractions of $\g$.

\begin{ex}
The adjoint and coadjoint representations of any $\BZ_2$-contraction can  be 
obtained as a special case of this construction.  Given $\sigma\in\Inv(\g)$, 
consider three involutions of $\tilde \g=\g\oplus\g$ as in Example~\ref{ex:h+h}:
\vskip.6ex
\centerline{
$\sigma_1(x,y)=(y,x)$, \ $\sigma_2(x,y)=(\sigma(x),\sigma(y))$, \ and 
$\sigma_3=\sigma_1\sigma_2$.}
\vskip.6ex\noindent
Then  $\tilde\g\langle\sigma_2\rangle^{\sigma_1}=\g\langle\sigma\rangle$ and
$\tilde\g\langle\sigma_2\rangle^{(\sigma_1)}_1$ is isomorphic to the adjoint module of
$\g\langle\sigma\rangle$; whereas 
$\tilde\g\langle\sigma_3\rangle^{\sigma_1}=\g\langle\sigma\rangle$ and
$\tilde\g\langle\sigma_3\rangle^{(\sigma_1)}_1$ is isomorphic to the coadjoint module of
$\g\langle\sigma\rangle$.
\end{ex}

{\bf Convention.} {\sl Whenever it is notationally convenient, we will expose our results in 
the most symmetric form, i.e., for an arbitrary permutation of indices $(10,01,11)$. 
Otherwise, we stick to our sample choice with $K=K_{01}$ and the $K$-module 
$V=\g_{10}{\,\propto\,}\g_{11}$.}

Let $(\ap,\beta,\gamma)$ be a permutation of $(10,01,11)$. The representation of
$K_\beta=G_{00}\ltimes N_\beta$ in $V_{\ap\gamma}:=\g_\ap{\,\propto\,}\g_\gamma$ is said to be a {\it degenerated isotropy representation}. 
There is a special situation in which the  algebra 
$\bbk[V_{\ap\gamma}]^{K_\beta}$ can explicitly be described.
Since $\bbk[V_{\ap\gamma}]^{K_\beta}=
(\bbk[V_{\ap\gamma}]^{N_{\beta}})^{G_{00}}$, 
the first step is to describe the algebra  $\bbk[V_{\ap\gamma}]^{N_{\beta}}$. 
It follows from the appropriate analogue of \eqref{eq:N01:V} that the dimension of any $N_{\beta}$-orbit in $V_{\ap\gamma}$ is at most 
$\dim\g_{\gamma}$ and if $p: V_{\ap\gamma}\to \g_\ap$ is the projection along 
$\g_\gamma$, then $\bbk[\g_{\ap}]\simeq p^*(\bbk[\g_{\ap}])\subset 
\bbk[V_{\ap\gamma}]^{N_{\beta}}$.

\begin{thm}   \label{thm:no-extra-inv}
Suppose that there is\/ $\tilde y\in \g_{\ap}$ such that $[\tilde y,\g_{\beta}]=\g_{\gamma}$. Then 
\begin{itemize}
\item[{\sf (i)}] \ $\bbk[V_{\ap\gamma}]^{N_{\beta}}\simeq\bbk[\g_{\ap}]$, 
\item[{\sf (ii)}] \ $\bbk[V_{\ap\gamma}]^{K_\beta}\simeq\bbk[\g_{\ap}]^{G_{00}}$
is a polynomial algebra, 
\item[{\sf (iii)}] \ $\bbk[V_{\ap\gamma}]$ is a free\/ 
$\bbk[V_{\ap\gamma}]^{K_\beta}$-module.
\end{itemize}
\end{thm}
\begin{proof}
(i)  
Since  $\g_{\gamma}\subset V_{\ap\gamma}$ is $N_{\beta}$-stable, the projection 
$\pi: V_{\ap\gamma}\to V_{\ap\gamma}/\g_{\gamma}\simeq \g_{\ap}$ is $N_{\beta}$-equivariant. Moreover, the induced
$N_{\beta}$-action on $V_{\ap\gamma}/\g_{\gamma}$ is trivial. In order to prove that
$\pi$ is the quotient by $N_{\beta}$, we use the Igusa lemma, see e.g. \cite[Lemma\,6.1]{rims07}.
Since $\pi$ is onto and $V_{\ap\gamma}/\g_{\gamma}$ is normal, it suffices to prove that generic fibres  of 
$\pi$ are $N_{\beta}$-orbits.
If $[\tilde y,\g_{\beta}]=\g_{\gamma}$ and we identify $\tilde y$ with $(\tilde y,0)\in V_{\ap\gamma}$, then the orbit
$N_{\beta}{\cdot}\tilde y=\tilde y+[\tilde y,\g_{\beta}]$ is precisely a fibre of 
$\pi$.
The appropriate analogue of Eq.~\eqref{eq:N01:V} shows that 
$\dim N_{\beta}{\cdot}(y',y'')=\dim [y',\g_{\beta}]$; 
in particular, it depends only on $y'\in\g_\ap$. 
Consequently, there is a dense open subset $\Omega\subset \g_{\ap}$ such
that $[y,\g_{\beta}]=\g_{\gamma}$ for all $y\in\Omega$. Finally,  identifying $\Omega$ with an open
subset in $V_{\ap\gamma}/\g_{\gamma}$, we see that $\pi^{-1}(y)$ is a sole orbit for all $y\in\Omega$.

(ii), (iii) This follows from (i) and  the corresponding properties of the algebra of invariants for
the isotropy representation $(G_{00}:\g_{\ap})$.
\end{proof}

The above proof shows that 
$\bbk[V_{\ap\gamma}]^{N_\beta}=p^*(\bbk[\g_\ap])$ if and only if 
there is\/ $\tilde y\in \g_{\ap}$ such that $[\tilde y,\g_{\beta}]=\g_{\gamma}$
if and only if $\max_{v\in V_{\ap\gamma}}\dim N_\beta{\cdot}v=\dim\g_\gamma$.

We will see in a moment that the hypothesis of Theorem~\ref{thm:no-extra-inv}
is equivalent to a coincidence of \CSS. 

\begin{lm}   \label{lm:raspred-equal}
For any quaternionic decomposition and 
$x\in\g_{\ap}$, we have $\dim [\g_{\beta},x]=\dim [\g_{\gamma},x]$.
\end{lm}
\begin{proof}
Set $d_{ij}=\dim \z_\g(x)_{ij}$. Applying \cite[Prop.\,5]{kr71} to one little and one big
$\BZ_2$-grading, we obtain:

$\begin{array}{rcrcl}  d_{\ap}-d_{00} & =& \dim\g_{\ap} &- &\dim\g_{00} \ , \\
        (d_{\ap}+d_{\gamma})-(d_{00}+d_{\beta}) & =& \dim(\g_{\ap}\oplus\g_{\gamma})& -
        &\dim(\g_{00}\oplus\g_{\beta}) \ .
        \end{array}$
\vskip1ex
\noindent
Taking the difference yields 
$d_{\gamma}-d_{\beta} = \dim\g_{\gamma}-\dim\g_{\beta}$, as required.
\end{proof}

\begin{prop}   \label{prop:rel-CSS}
There exists $x\in\g_{\ap}$ such that
$[\g_{\beta},x]=\g_{\gamma}$ 
if and only if any \CSS\ $\ce_{\ap}\subset \g_{\ap}$  is also a \CSS\ in
$\g_{\ap}\oplus\g_{\gamma}$;
\end{prop}
\begin{proof}
``$\Rightarrow$''. The set of elements $x\in\g_{\ap}$ having such property is open.
If $x$ is $G_{00}$-regular and semisimple, then $\z_\g(x)_{\ap}$ is a \CSS\ in $\g_{\ap}$.
By Lemma~\ref{lm:raspred-equal}, the assumption implies that
$\z_\g(x)_{\gamma}=\{0\}$. Thus, the \CSS\ $\ce_{\ap}:=\z_\g(x)_{\ap}$ has the
property that  $\z_\g(\ce_{\ap})\cap (\g_{\ap}\oplus\g_{\gamma})
=\z_\g(x)\cap (\g_{\ap}\oplus\g_{\gamma})=
\ce_{\ap}$, i.e.,
$\ce_{\ap}$ is a \CSS\ in $\g_{\ap}\oplus\g_{\gamma}$.
 
``$\Leftarrow$''.  Reversing the preceding argument shows that any $G_{00}$-regular
semsimple 
element $x\in\ce_{\ap}$ has the required property.
\end{proof}

\noindent
This proposition yields a simple necessary condition for coincidence: 
\begin{cor}
If $\ce_\ap$ is also a \CSS\ in
$\g_\ap\oplus\g_\gamma$, then $\dim\g_\beta\ge \dim\g_\gamma$.
\end{cor}

\noindent
Combining our previous results, we get an explicit description of the algebra of 
invariants for  certain degenerated  isotropy representations.

\begin{thm}   \label{thm:inv-degen}
Let $\g$ be a semisimple Lie algebra and  $\{\sigma_1,\sigma_2,\sigma_3\}$  a triple of involutions of $\g$ such that $\sigma_1\sigma_2=\sigma_3$.
\begin{itemize}
\item[{\sf (i)}] \ Suppose that $\sigma_1$ is of  maximal rank. Then

$\bbk[\g_{11}{\,\propto\,}\g_{01}]^{G_{00}\ltimes N_{10}}\simeq
\bbk[\g_{11}]^{G_{00}}$ \ and \ 
$\bbk[\g_{10}{\,\propto\,}\g_{01}]^{G_{00}\ltimes N_{11}}\simeq
\bbk[\g_{10}]^{G_{00}}$.

\item[{\sf (ii)}] \ Suppose that $\{\sigma_1,\sigma_2\}$ is a dyad. Then

$\bbk[\g_{11}{\,\propto\,}\g_{10}]^{G_{00}\ltimes N_{01}}\simeq
\bbk[\g_{11}]^{G_{00}}$ \ and \ 
$\bbk[\g_{11}{\,\propto\,}\g_{01}]^{G_{00}\ltimes N_{10}}\simeq
\bbk[\g_{11}]^{G_{00}}$.

\item[{\sf (iii)}]  \ Suppose that $\{\sigma_1,\sigma_2,\sigma_3\}$ is a triad, then 
$\bbk[\g_{\ap}\ltimes\g_{\gamma}]^{G_{00}\ltimes N_{\beta}}\simeq 
\bbk[\g_{\ap}]^{G_{00}}$ for all permutations $(\ap,\beta,\gamma)$ of\/ 
$(01,10,11)$.
\end{itemize}
\noindent
In all previous cases, $\bbk[\g_{\ap}\ltimes\g_{\gamma}]^{N_\beta}=\bbk[\g_\ap]$ and\/ $\bbk[\g_{\ap}\ltimes\g_{\gamma}]$ is a free\/ 
$\bbk[\g_{\ap}\ltimes\g_{\gamma}]^{G_{00}\ltimes N_{\beta}}$-module.
\end{thm}
\begin{proof}
(i) \ Combine Theorem~\ref{thm:sovpad1}, Theorem~\ref{thm:no-extra-inv}, and  Proposition~\ref{prop:rel-CSS}.
\\
(ii) \ Combine Theorem~\ref{thm:sovpad2}, Theorem~\ref{thm:no-extra-inv}, and  Proposition~\ref{prop:rel-CSS}.
\\
(iii) \ Combine Corollary~\ref{cor:triada}, Theorem~\ref{thm:no-extra-inv}, and  Proposition~\ref{prop:rel-CSS}.
\end{proof}

If $[\g_{\beta},x]$ is a proper subspace of $\g_{\gamma}$ for all $x\in\g_{\ap}$, i.e.,
the maximal dimension of $N_{\beta}$-orbits in $V_{\ap\gamma}$  is less than $\dim\g_{\gamma}$, then 
 a general description of $\bbk[V_{\ap\gamma}]^{N_{\beta}}$ 
(and hence of $\bbk[V_{\ap\gamma}]^{K_\beta}$) is not available.
Nevertheless, the $K_\beta$-module $V_{\ap\gamma}$ always retains some good properties of initial isotropy representation of a reductive group.

Below, we consider the problem of existence of a generic stabiliser for the linear 
action of an algebraic group $K$ on $V$.  Recall that a subalgebra $\q\subset \ka$ is a 
{\it generic stabiliser\/}  if there exists
a dense open subset $\Omega\subset V$ such that the stabiliser 
$\ka^v=\{ k\in \ka\mid k{\cdot}v=0\}$ 
is $K$-conjugate to $\q$ for all $v\in\Omega$. We then write $\q=\sgp(K{:}V)$ and the elements  
of $\Omega$ are called $K$-{\it generic points} (or just {\it generic points\/} 
if the group is clear  from the context).
A generic stabiliser always exists for the {\sf reductive\/} group actions on 
{\sf smooth} affine varieties, see e.g. \cite[\S\,7]{t55}. Moreover, 
for the isotropy representation $(G_0:\g_1)$, 
$x\in\g_1$ is $G_0$-generic {\sl if and only if\/} $x$ is $G_0$-regular and semisimple 
{\sl if and only if\/} $\z_\g(x)\cap\g_1$ is a \CSS\ in $\g_1$.

In the following theorem, we write $\g^\xi$ in place of $\z_\g(\xi)$ ($\xi\in\g$), and 
$\g^\xi_{ij}:=\g^\xi\cap\g_{ij}$.

\begin{thm}  \label{thm:sgp-i-pgp}
{\sf (i)} \ The degenerated isotropy representation $(K=G_{00}\ltimes N_{01}:V=\g_{10}{\,\propto\,}\g_{11})$ always has 
a generic stabiliser. More precisely, 
let $\xi\in\g_{10}$  be a $G_{00}$-generic point for the (little isotropy) representation
$(G_{00}:\g_{10})$ and let $\eta$ be a generic point in the $G_{00}^\xi$-module
$\g_{11}^\xi$.
Then $v=(\xi, \eta)$
is a $K$-generic point in  $V$   and
$\sgp(K:V)=\sgp(G_{00}^\xi:\g_{11}^\xi) \ltimes (\g_{01}^\xi)^a$.

{\sf (ii)} \ $\trdeg\bbk(V)^{K}=\dim \g_{10}\md G_{00}+ 
\dim \g_{11}^\xi\md G_{00}^\xi$.

{\sf (iii)} \ $\bbk(V)^{K}$ is the fraction field of\/ $\bbk[V]^K$.
\end{thm}
\begin{proof} (i) \ 
As explained above, $\xi$ is semisimple and $\ce_{10}:=\g^\xi_{10}$ is a \CSS\ in $\g_{10}$.
Then 
\beq   \label{eq:dir-sum}
\g_{11}=[\g_{01},\xi]\oplus\g^\xi_{11} .
\eeq  
Note that $\g^\xi_{00}\oplus\g^\xi_{11}$ is a $\BZ_2$-grading of the reductive Lie algebra
$\g^{\sigma_3}\cap\g^\xi$.  Let $\tilde\ce$ \ be a \CSS\ in $\g^\xi_{11}$ and 
let $\eta\in \tilde\ce$ be a $G_{00}^\xi$-generic point. The stabiliser of $(\xi,\eta)$ in
$\ka$ is determined by the following conditions:
\[
\ka^{(\xi,\eta)}=\{(s_0,s_1)\in \g_{00}\ltimes\g_{01}^a \mid [s_0,\xi]=0 \ \& 
\ [s_0,\eta]+[s_1,\xi]=0\} .
\]
Hence $s_0\in \g^\xi_{00}$. Then the second equation and \eqref{eq:dir-sum} 
show that $[s_0,\eta]=0$ and $[s_1,\xi]=0$; that is, $s_0\in (\g^\xi_{00})^\eta$ and 
$s_1\in \g^\xi_{01}$. Thus,
\[
    \ka^{(\xi,\eta)}=(\g^\xi_{00})^\eta\ltimes (\g^\xi_{01})^a=
    \sgp(G_{00}^\xi:\g_{11}^\xi) \ltimes (\g_{01}^\xi)^a .
\]
Note that this stabiliser does not essentially depend on $(\xi,\eta)$. 
If $\xi'\in \ce_{10}$  and 
$\eta'\in \tilde\ce$ are some other generic points, then 
$\ka^{(\xi,\eta)}=\ka^{(\xi',\eta')}$. In view of this observation, to prove that
$\ka^{(\xi,\eta)}$ is a generic stabiliser, it is sufficient to show that the $K$-saturation
of $\ce_{10}{\,\propto\,}\tilde\ce\subset \g_{10}{\,\propto\,} \g_{11}$ is dense. 

By our constructions, there are the following relations for
$\ce_{10}$ and $\tilde\ce$:

{\bf --} \quad $G_{00}^\xi{\cdot}\tilde\ce$ is dense in $\g_{11}^\xi$ ;

{\bf --} \quad $N_{01}{\cdot}(\xi',\eta')=(\xi', [\g_{01},\xi']+\eta')$ and $ [\g_{01},\xi']= [\g_{01},\xi]$ if $\xi'\in\ce_{10}$ is generic;

{\bf --} \quad $G_{00}{\cdot}\ce_{10}$ is dense in $\g_{10}$.

\noindent
The first two relations and the fact that $G_{00}^\xi$ does not affect
$\ce_{10}=\g_{11}^\xi$ show us that
$(G_{00}^\xi\ltimes N_{01}){\cdot}(\ce_{10}{\,\propto\,} \ce)$ is dense in 
$\ce_{10}{\,\propto\,}\g_{11}$. Then the last relation implies that
$G_{00}{\cdot}(\ce_{10}{\,\propto\,}\g_{11})$ is dense in $V$.

(ii) \ 
By (i) and the Rosenlicht theorem \cite[2.3]{t55}, we have 
\begin{multline*}
  \dim\sgp(K{:}V)=\dim \g_{01}^\xi+\dim \sgp(G_{00}^\xi:\g_{11}^\xi) = \\
  \dim\g_{01}^\xi+\dim\g_{00}^\xi-\dim\g_{11}^\xi
+\trdeg \bbk(\g_{11}^\xi)^{G_{00}^\xi} .
\end{multline*}
Since $\g_{11}^\xi$ is an orthogonal $G_{00}^\xi$-module, it follows from  \cite{lu72}
that the field $\bbk(\g_{11}^\xi)^{G_{00}^\xi}$ is the fraction field of the algebra
$\bbk[\g_{11}^\xi]^{G_{00}^\xi}$.
Hence $\trdeg \bbk(\g_{11}^\xi)^{G_{00}^\xi}=\dim \g_{11}^\xi\md G_{00}^\xi$.
Then using again the Rosenlicht theorem, we obtain
\begin{multline*}
 \trdeg\bbk(V)^{K}\stackrel{\text{Rosenlicht}}{=}\dim V-\dim K+\dim \sgp(K:V){=}\\
(\un{\dim\g_{11}}+\dim\g_{10}){-}(\un{\dim\g_{01}}+\dim\g_{00}){+}
(\un{\dim\g_{01}^\xi}+
\dim\g_{00}^\xi-\un{\dim\g_{11}^\xi}+\dim \g_{11}^\xi\md G_{00}^\xi)= \\
(\dim\g_{10}-\dim\g_{00}+\dim\g_{00}^\xi)+\dim \g_{11}^\xi\md G_{00}^\xi=
\dim \g_{10}\md G_{00}+\dim \g_{11}^\xi\md G_{00}^\xi ,
\end{multline*}
where the underlined terms in the second line are cancelled out, in view of 
Lemma~\ref{lm:raspred-equal} applied to $\xi\in\g_{10}$.

(iii) \  The proof of  \cite[Lemma\,2.6]{coadj07} applies in this situation.
\end{proof}

\begin{cor}
The following conditions are equivalent:
{\sf (i)} \  $\trdeg \bbk(V)^K=\dim\g_{10}\md G_{00}$, 
{\sf (ii)} \ $\g_{11}^\xi=0$ for almost all $\xi\in\g_{10}$, 
{\sf (iii)} \ any \CSS\ $\ce_{10}\subset\g_{10} $ is also a \CSS\ in $\g_{10}\oplus\g_{11}$,
{\sf (iv)} \ $\bbk[V]^K\simeq \bbk[\g_{10}]^{G_{00}}$. 
\end{cor}
\begin{cor}   \label{cor:4.12}
It is always true that\/ $\trdeg \bbk(V)^K=\dim \g_{1\star}\md G_{0\star}=
\trdeg\bbk(\g_{1\star})^{G_{0\star}}$.
\end{cor}
\begin{proof}
Choose a \CSS\ $\ce_{10}\subset\g_{10}$.
By \cite{kr71},   $\dim\g_{10}\md G_{00}=\dim\ce_{10}$ and 
$\dim \g_{11}^\xi\md G_{00}^\xi$ is the dimension of any \CSS\ in $\g_{11}^\xi$.
Without loss of generality, we may assume that $\xi\in \ce_{10}$ and
$\z_\g(\xi)=\z_\g(\ce_{10})$.  Therefore, if $\h$ is a Cartan subspace in
$\g_{11}^\xi=\z_\g(\ce_{10})_{11}$, then $\ce_{10}\oplus\h$ is a \CSS\ in 
$\g_{1\star}$. That is, the right-hand side in Theorem~\ref{thm:sgp-i-pgp}(ii) is exactly 
$\dim \g_{1\star}\md G_{0\star}$. The second equality is a general property of orthogonal
representations of reductive groups \cite{lu72}.
\end{proof}

There is a general `contraction procedure' for obtaining $K$-invariants in $\bbk[V]$ or $\bbk[V^*]$.
Recall that $V$ and $V^*$ coincide as vector spaces and are isomorphic as $G_{00}$-modules. Only the $N_{01}$-actions on them are different.
Therefore, $\bbk[V]^K$ and $\bbk[V^*]^K$  can be regarded as  subalgebras of 
$\bbk[\g_{10}\oplus\g_{11}]^{G_{00}}$.
Let $\bbk[\g_{10}\oplus\g_{11}]_{(a,b)}$ denote the space of bi-homogeneous polynomials
of degree $a$ with respect to $\g_{10}$ and degree $b$ with respect to $\g_{11}$.
Any homogeneous polynomial $f\in \bbk[\g_{10}\oplus\g_{11}]$ of degree 
$n$ can be decomposed into the sum of bi-homogeneous components 
$f=\sum_{i=k}^m f_i$, where $f_i\in \bbk[\g_{10}\oplus\g_{11}]_{(n-i,i)}$.
Assuming that $f_k, f_m\ne 0$, we set
$f^\bullet:=f_k$ and $f_\bullet:=f_m$. That is, $f^\bullet$ is the bi-homogeneous component 
having the maximal degree w.r.t. $\g_{10}$.
 
\begin{prop}  \label{z2-contra}
Suppose that $f\in \bbk[\g_{1\star}]^{G_{0\star}}$ is homogeneous. Then
\vskip1ex
\centerline{$f^\bullet\in \bbk[\g_{10}{\,\propto\,}\g_{11}]^K$
and 
$f_\bullet\in \bbk[\g_{11}{\,\propto\,}\g_{10}]^K$.}
\end{prop}
\begin{proof}
The proof is similar to that of \cite[Prop.\,3.1]{coadj07}, which was designed
for the adjoint and coadjoint representations of $K$.
For convenience, we repeat it here. Clearly, each bi-homogeneous component $f_i$ of $f$ is
$G_{00}$-invariant. For $x\in\g_{01}$, let $\eus D_x$ denote the corresponding 
derivation of $\bbk[\g_{1\star}]$. 
(Here we regard the space $\g_{1\star}$ as $G_{0\star}$-module, i.e.,
$x$ is an element of the Lie algebra $\g_{0\star}$.)
Since $\ad x(\g_{10})\subset \g_{11}$ and $\ad x(\g_{11})\subset \g_{10}$,
we have 
\[
    \eus D_x \bigl(\bbk[\g_{1\star}]_{(a,b)}\bigr) 
    \subset \bbk[\g_{1\star}]_{(a+1,b-1)}  \oplus \bbk[\g_{1\star}]_{(a-1,b+1)}.
\]
Of course, if $a=0$ or $b=0$, then the summand with index $a-1$ or $b-1$ does not occur
in the right-hand side.
One can write
$\eus D_x=\eus D_x^+ +\eus D_x^-$, where 
\[
\eus D_x^+ \bigl(\bbk[\g_{1\star}]_{(a,b)}\bigr)  \subset \bbk[\g_{1\star}]_{(a+1,b-1)}  \text{  and  } 
\eus D_x^- \bigl(\bbk[\g_{1\star}]_{(a,b)}\bigr)  \subset \bbk[\g_{1\star}]_{(a-1,b+1)} .
\]
Since $\eus D_x(f)=0$, we have $\eus D_x^+(f^\bullet)=0$ and $\eus D_x^-(f_\bullet)=0$.
It also follows from \eqref{eq:N01:V}  that 
$\eus D_x^+$ (resp. $\eus D_x^-$) is the derivation corresponding to $x$, as an element
of $\g_{01}^a\subset \ka$,
in the algebra $
\bbk[\g_{01}{\,\propto\,}\g_{11}]$ 
(resp. $
\bbk[\g_{11}{\,\propto\,}\g_{10}]$).
\end{proof}

\begin{cor}
The algebras\/ $\bbk[V]^K$ and\/  $\bbk[V^*]^K$ are bi-graded.
\end{cor}
\begin{proof} 
The previous proof shows that
if $f\in \bbk[V]^K$, then also $f^\bullet \in \bbk[V]^K$. Then one considers $f-f^\bullet$, etc.
\end{proof}

\begin{rmk}   \label{rmk:bullet}
Set $\gr^\bullet (\bbk[\g_{1\star}]^{G_{0\star}})=\{f^\bullet \mid f\in \bbk[\g_{1\star}]^{G_{0\star}}\}\subset \bbk[V]^K$.
We already know that $\trdeg \bbk(V)^K =\trdeg\bbk(\g_{1\star})^{G_{0\star}}=\dim (\g_{1\star}\md G_{0\star})$ and
$\bbk(V)^K$ is the fraction field of $\bbk[V]^K$. Therefore,
one might expect that $\gr^\bullet (\bbk[\g_{1\star}]^{G_{0\star}})=\bbk[V]^K$ 
 in good cases. Indeed, it often happens  that there is a set of basic invariants
 $f_1,\dots,f_m\in \bbk[\g_{1\star}]^{G_{0\star}}$ such that 
 $f_1^\bullet, \dots,f_m^\bullet$ are algebraically independent.
But, unlike the case of the coadjoint representation of $K$, this does not guarantee that
$f_1^\bullet, \dots,f_m^\bullet$ generate $\bbk[V]^K$ (cf. \cite[Theorem\,4.2]{coadj07}).  
True generators of $\bbk[V]^K$
can have smaller degrees. 
The reason is that, for the $K$-module $V$, the complement of the set of $K$-regular
points may contain a divisor. (By \cite[Theorem\,3.3]{coadj07}, this cannot happen for 
the coadjoint  representation of a $\BZ_2$-contraction.)
\end{rmk}

\section{Invariants of certain degenerated isotropy representations}  
\label{sect5}

\noindent
We continue to work with a quaternionic decomposition~\eqref{eq:quatern_matrix}.
In Section~\ref{sect4}, we considered a good situation  in which a \CSS\ 
$\ce_{10}\subset\g_{10}$ is also a \CSS\ in $\g_{1\star}$. This is equivalent to that 
$\z_\g(\ce_{10})\cap\g_{1\star}=\ce_{10}$. Using Theorem~\ref{thm:no-extra-inv},
one then obtains a  description of $\bbk[\g_{10}{\,\propto\,}\g_{11}]^{N_{01}}$, etc.
In this section, we consider a less restrictive condition that $\ce_{10}$ contains 
$G_{0\star}$-regular elements of $\g_{1\star}$. In other words,
\vskip1ex
\hbox to \textwidth{\enspace \refstepcounter{equation} (\theequation) \label{eq:less-restr}
\hfil 
$\z_\g(\ce_{10})\cap\g_{1\star}$ is a \CSS\ in $\g_{1\star}$.
\hfil
}
\vskip.8ex\noindent
Then $\de:=\z_\g(\ce_{10})\cap\g_{1\star}=\ce_{10}\oplus \ah$ for some 
subspace $\ah \subset\g_{11}$. Recall that the generalised  Weyl group for 
the \CSS\ 
$\de\subset \g_{1\star}$ is 
$\tilde W=N_{0\star}(\de)/Z_{0\star}(\de)$, where
$N_{0\star}(\de)=\{ s\in G_{0\star}\mid s{\cdot}\de\subset \de\}$ and
$Z_{0\star}(\de)=\{ s\in G_{0\star}\mid s{\cdot}x=x \ \forall x\in\de\}$. Set $m=\dim\de$.
\\
Obviously, the involution $\sigma_2$ preserves  $\de$ and the corresponding
eigenspaces are $\ce_{10}$ and $\ah$.
\begin{lm}
$\bar\sigma_2=\sigma_2\vert_{\de}$ normalises the group $\tilde W\subset
GL(\de)$, i.e., $\bar\sigma_2\tilde W\bar\sigma_2^{-1}=\tilde W$.
\end{lm}
\begin{proof}
If $g\in G_{0\star}$, then $\sigma_2(g)\in G_{0\star}$ as well, and 
$\sigma_2(g)$ acts on $\g$ as composition $\sigma_2\circ g\circ \sigma_2^{-1}$.
Therefore, if  $g\in N_{0\star}(\de)$, then also $\sigma_2(g)\in N_{0\star}(\de)$.
\end{proof}

Set $X=\ov{G_{0\star}{\cdot}\ce_{10}}\subset \g_{1\star}$.  Since $\ov{G_{00}{\cdot}\ce_{10}}=\g_{10}$, we may as well define $X$ as $\ov{G_{0\star}{\cdot}\g_{10}}$.

\begin{thm}  \label{thm:ci}
If condition~\eqref{eq:less-restr}  is satisfied,
then  $X$ is an irreducible 
complete intersection in $\g_{1\star}$ and the ideal of 
$X$ in $\bbk[\g_{1\star}]$, $\eus I(X)$, is generated by some basic invariants in\/ $\bbk[\g_{1\star}]^{G_{0\star}}$.
\end{thm}
\begin{proof}
Clearly, $X$ is irreducible.  Since $\ce_{10}$ contains $G_{0\star}$-regular elements
of $\g_{1\star}$ and $\ov{G_{0\star}{\cdot}\de}=\g_{1\star}$,
we have $\dim \g_{1\star}-\dim X=\dim{\de}-\dim\ce_{10}=\dim\ah$.

We know that $\bar\sigma_2\in GL(\de)$ is of finite order, normalises $\tilde W$, and
has regular eigenvectors (in $\ce_{10}$).
By \cite[6.3--6.5]{sp74}, this implies the following:

{\bf -- } \ The centraliser $W_{10}$ of $\bar\sigma_2$ in $\tilde W$ is reflection group 
in the $\bar\sigma_2$-eigenspace $\ce_{10}$.

{\bf -- } \ $\bbk[\ce_{10}]^{W_{10}}$ is generated by the restrictions to $\ce_{10}$ of some 
basic invariants in $\bbk[\de]^{\tilde W}$. Namely, a set of basic invariants
$\{f_1,\dots,f_m\}$ can be chosen such that each $f_i$ is a $\bar\sigma_2$-eigenvector,
say $\bar\sigma_2(f_1)=\esi_i f_i$, where $\esi_i\in\{1,-1\}$. Then 
$\bbk[\ce_{10}]^{W_{10}}$ is freely generated by $f_1\vert_{\ce_{10}}$ such that 
$\esi_1=1$. In particular, the restriction homomorphism
$\bbk[\de]^{\tilde W}\to \bbk[\ce_{10}]^{W_{10}}$  is onto.

{\bf -- } \ The eigenvalues of $\bar\sigma_2$ in $\de $ are $\esi_1,\dots,\esi_m$.

\noindent 
It follows that there is a set of basic invariants $\{f_1,\dots,f_m\}$ in $\bbk[\de]^{\tilde W}$
such that $f_{1},\dots, f_k$ vanish on $\ce_{10}$, whereas the restrictions of 
$f_{k+1},\dots, f_m$ to $\ce_{10}$ freely generate $\bbk[\ce_{10}]^{W_{10}}$. Moreover, here
$m-k=\dim\ce_{10}$ and $k=\dim\ah$.

Since $\bbk[\g_{1\star}]^{G_{0\star}}\simeq\bbk[\de]^{\tilde W}$,  there
are also $k$ basic invariants in $\bbk[\g_{1\star}]^{G_{0\star}}$ that vanish on $X$. 
Let $F_i$ denote the $G_{0\star}$-invariant corresponding to $f_i$.
Then $F_1,\dots,F_m$ is a regular sequence in $\bbk[\g_{1\star}]$ and
$\eus V(F_{1},\dots,F_k)$ is an unmixed variety of codimension $\dim\ah$.
Furthermore, $X\subset \eus V(F_{1},\dots,F_k)$ and $\codim X=k$.
Hence $X$ is an irreducible component of $\eus V(F_{1},\dots,F_k)$.
Our goal is to prove that the ideal $(F_{1},\dots,F_k)$ is prime.
The first step is prove that $\eus V(F_{1},\dots,F_k)$ is irreducible.

Recall that the morphism 
$\pi: \g_{1\star}\to \g_{1\star}\md G_{0\star}$ is equidimensional, and 
a generic fibre of $\pi$ is irreducible, since it is a (closed) $G_{0\star}$-orbit.
If $Z$ is a $G_{0\star}$-stable closed subset of $\g_{1\star}$, then 
$\pi(Z)$ is also closed \cite[4.4]{t55}.
The previous discussion on invariants of $\tilde W$ and $W_{10}$ shows that
$\pi(\eus V(F_{1},\dots,F_k))=
\pi(X)=\pi(\ce_{10})$ is an affine space of dimension $m-k$. 

Let $Y$ be an irreducible component of  $\eus V(F_{1},\dots,F_k)$.
Since $\pi$ is equidimensional, $\dim X=\dim Y$, and $\pi(Y)\subset
\pi(\eus V(F_{1},\dots,F_k))=\pi(X)$, 
we  must have $\pi(X)=\pi(Y)$. Moreover, for
generic $\eta=\pi(x)\in \pi(Y)$ ($x\in\ce_{10}$), we have $\pi^{-1}(\eta)=
G_{0\star}{\cdot}x\subset X$. Thus, $Y=X$ and $X=\eus V(F_{1},\dots,F_k)$ is irreducible.

By \cite[Lemma\,4]{ko63}, the primeness of 
the ideal $(F_{1},\dots,F_k)$ will follow from the fact that there is a point
$y\in X=\eus V(F_{1},\dots,F_k)$ such that the differentials
$\textsl{d}F_{1},\dots,\textsl{d}F_{k}$ are linearly independent at $y$.
Write $(\textsl{d}F)_y$ for the value of $\textsl{d}F$ at $y$.
By assumption, $X$ contains a $G_{0\star}$-regular semisimple element $x$, and it is known 
that even $(\textsl{d}F_{1})_x,\dots,(\textsl{d}F_{m})_x$ are linearly independent \cite{kr71}.
\end{proof}

\begin{rmk}
Using the fact that 
$(\textsl{d}F_{1})_x,\dots,(\textsl{d}F_{k})_x$ are linearly independent for any 
$G_{0\star}$-regular point of $x\in X$, one easily proves that  $X$ is non-singular in codimension 2 and therefore
$X$ is normal. (A similar argument in the context of adjoint representations is found in
\cite[\S\,5]{ri87}.)
\end{rmk}
Keep the previous notation, i.e.,   $m=\rk(G/G_{0\star})$,
$X=\ov{G_{0\star}{\cdot}\ce_{10}}$, and  $F_1,\dots,F_m \in 
\bbk[\g_{1\star}]^{G_{0\star}}$ are the
basic invariants such that
$\eus I(X)=(F_{1},\dots,F_k)$ and  $\bbk[X]^{G_{0\star}}=\bbk[F_{k+1},\dots,F_m]$.
Recall that $m-k=\dim\ce_{10}$ and $k=\dim\ah$.
We wish to describe the invariants  of the isotropy
representation $(K:V=\g_{10}{\,\propto\,}\g_{11})$, which is a contraction of 
$(G_{0\star}:\g_{1\star})$. 

As in Section~\ref{sect4}, we consider  bi-homogeneous components of $F_i$ with respect 
to the sum $\g_{1\star}=\g_{10}\oplus\g_{11}$. Recall that $F_i^\bullet$ stands for
the bi-homogeneous component of maximal degree with respect to $\g_{10}$, and 
$F_i^\bullet\in \bbk[\g_{10}{\,\propto\,}\g_{11}]^{G_{00}\ltimes N_{01}}$
by Prop.~\ref{z2-contra}.

\begin{thm}   \label{thm:main5}
Suppose that condition~\eqref{eq:less-restr}  is satisfied.
\\
{\sf (i)} \  The field\/ $\bbk(\g_{10}{\,\propto\,}\g_{11})^{N_{01}}$ is  
generated  by the  coordinate functions
on $\g_{10}$ and the polynomials $F_i^\bullet$, $i=1,\dots,k$.
\\
{\sf (ii)} \ Suppose that 
$(\textsl{d}F_{1})_e,\dots,(\textsl{d}F_{k})_e$ are linearly independent for
a $G_{00}$-regular nilpotent element $e\in\g_{10}$.  Then 

{\bf --} \ the algebra  $\bbk[\g_{10}{\,\propto\,}\g_{11}]^{N_{01}}$ is freely generated by the  
coordinate functions on $\g_{10}$ and the polynomials $F_i^\bullet$, $i=1,\dots,k$;

{\bf --} \  the algebra $\bbk[\g_{10}{\,\propto\,}\g_{11}]^{G_{00}\ltimes N_{01}}$ is freely generated by 
the basic invariants in $\bbk[\g_{10}]^{G_{00}}$ and the polynomials $F_i^\bullet$, $i=1,\dots,k$.
\end{thm}
\begin{proof}
We are only interested in basic invariants $F_i$ with $i\le k$. Let $\deg F_i=n_i$. 
Since $F_i$ vanishes on $X$ and $\g_{10}=\ov{G_{00}{\cdot}\ce_{10}}\subset X$, the 
bi-homogeneous component of degree $(n_i,0)$ is trivial. 

By~\eqref{eq:less-restr}, $\g_{10}$ contains $G_{0\star}$-regular semisimple elements.
If $s\in\g_{10}$ is such an element, then $(\textsl{d}F_{1})_s,\dots,(\textsl{d}F_{k})_s$ are linearly independent.  The fact that $(\textsl{d}F_{i})_s$ is nonzero for some $s\in \g_{10}$
implies that each $F_i$ has a bi-homogeneous
component of degree $(n_i-1,1)$, which thereby is equal to  $F_i^\bullet$.
Since this bi-homogeneous component is linear with respect to $\g_{11}$, it
can be written as
\beq   \label{eq:F-bullet}
    F_i^\bullet(y_0,y_1)= \kappa( \cF_i(y_0),y_1) ,
\eeq
where $\cF_i :  \g_{01}\to \g_{11}$ is a polynomial mapping of degree $n_i-1$.
Then 
\[
(\textsl{d}F_{i})_{y_0}=\cF_i(y_0) .
\]
As $F_i^\bullet$ is $G_{00}$-invariant, the mapping $\cF_i$ must be $G_{00}$-equivariant.
It follows from Equations~\eqref{eq:F-bullet} and \eqref{eq:N01:V} that
the $N_{01}$-invariance of $\cF^\bullet_i$ is equivalent to that
$\cF_i(y_0)$ commutes with $y_0$. 
Recall that $\g_{11}$ is a $K$-stable subspace of $V=\g_{10}{\,\propto\,}\g_{11}$
and the induced action of $N_{01}$ on $V/\g_{11}\simeq \g_{10}$ is trivial.
Consider the polynomial mapping 
\[
  \psi:\g_{10}{\,\propto\,}\g_{11} \to (V/\g_{11})\times \bbk^{k}\simeq \g_{10} \times \bbk^{k},
\]
defined by 
\[
 \psi(y_0,y_1)=(y_0, F^\bullet _{1}(y_0,y_1),\dots,F^\bullet _k(y_0,y_1))=
                        \bigl(y_0, \kappa(\cF_1(y_0),y_1),\dots,\kappa(\cF_k(y_0),y_1)\bigr) .
\]
Clearly the $N_{01}$-action on $(V/\g_{11})\times \bbk^{k}$ is trivial and
$\psi$ is $N_{01}$-equivariant. Our ultimate goal is to prove that $\psi$ is the quotient
 by $N_{01}$. But we can perform it  under the additional hypothesis in part (ii).
 Without extra hypotheses, we can only get the result on the field of $N_{01}$-invariants.

(i) \ Let $\Omega_s\subset \g_{10}$ be the open subset of semisimple elements that are both
$G_{0\star}$-regular and $G_{00}$-regular.
Condition~\eqref{eq:less-restr} means that $\Omega_s\ne\varnothing$. 
Since the vectors $\cF_1(\zeta),\dots,\cF_k(\zeta)$ are linearly independent
for all $\zeta\in\Omega_s$, we have $\Omega_s\times \bbk^{k} \subset \Ima\psi$. Hence $\psi$ 
is dominant. For $(y_0,z_{1},\dots,z_k)\in \g_{10}\times \bbk^{k}$, the fibre 
\[
   \psi^{-1}(y_0,z_{1},\dots,z_k)=\{ (y_0, y_1)\mid \kappa(\cF_j(y_0), y_1)=z_j, \
   1\le j\le k \} 
\]
is an $N_{01}$-stable affine subspace of $\g_{10}{\,\propto\,}\g_{11}$.
Furthermore, if  $y_0\in\Omega_s$, then
\[
 \dim \psi^{-1}(y_0,z_{1},\dots,z_k)=\dim\g_{11}- k . 
\]
On the other hand,
if $(y_0,y_1)\in \psi^{-1}(y_0,z_{1},\dots,z_k)$, then 
\[
 N_{01}{\cdot}(y_0,y_1)=
(y_0, y_1+ [\g_{01},y_0])\subset \psi^{-1}(y_0,z_{1},\dots,z_k)
\]
and $\dim [\g_{01},y_0]=\dim [\g_{11},y_0]=\dim\g_{11}-\dim \z_\g(y_0)_{11}$.
If $y_0\in\Omega_s$, then $\z_\g(y_0)_{10}$ is a \CSS\ in $\g_{10}$ and
$\z_\g(y_0)_{1\star}$ is a \CSS\ in $\g_{1\star}$. 
Hence $\dim \z_\g(y_0)_{11}=\dim\ah=k$ and
\[
\dim N_{01}{\cdot}(y_0,y_1)=\dim \g_{11}-k=\dim \psi^{-1}(y_0,z_{1},\dots,z_k) .
\]
As the orbits of a unipotent group on affine varieties are closed 
\cite[Theorem\,2]{ros61}
and isomorphic to affine spaces, we obtain
$\psi^{-1}(y_0,z_{1},\dots,z_k)=N_{01}{\cdot}(y_0,y_1)$ for $y_0\in\Omega_s$, 
i.e., almost all fibres of $\psi$ are just $N_{01}$-orbits.

By \cite[Lemma\,2.1]{t55}, this means that  the coordinates on 
$\g_{10}$ and $F_1^\bullet,\dots,F_k^\bullet$ generate the field 
$\bbk(\g_{10}{\,\propto\,}\g_{11})^{N_{01}}$.

(ii) \ 
We wish is to apply the Igusa lemma \cite[Lemma\,6.1]{rims07} to $\psi$. 
As the target variety $\g_{10}\times\bbk^k$ is normal and generic fibres of $\psi$ are $N_{01}$-orbits, it remains to check that $\Ima\psi$ contains an open subset of 
$\g_{10}\times\bbk^k$ whose complement is of codimension at least $2$.

Let $\widetilde\Omega\subset\g_{10}$ be the open set of all $G_{00}$-regular elements 
(i.e., not only semisimple ones!).
It follows from \cite{kr71} that  $\codim(\g_{10}\setminus\widetilde\Omega)\ge 2$. 
Hence it would be sufficient 
to have that $\widetilde\Omega\times\bbk^k \subset\Ima\psi$.
This, in turn, would follow from the fact that 
\vskip.7ex
\hbox to \textwidth{\enspace ($\Diamond$) \hfil
the differentials 
$(\textsl{d}F_{1})_\eta,\dots,(\textsl{d}F_{k})_\eta$ are linearly independent for all
$\eta\in\widetilde\Omega$. \hfil }

An obstacle is that if $\eta\in\widetilde\Omega$, then $\eta$ is not necessarily
$G_{0\star}$-regular. Therefore, we need the assumption that
$(\textsl{d}F_{1})_e,\dots,(\textsl{d}F_{k})_e$ are linearly independent
for a nilpotent $e\in\widetilde\Omega$. 
Then a standard deformation argument guarantees us condition ($\Diamond$).

Thus, the assumptions of the Igusa lemma are satisfied and the assertion on the algebra 
of $N_{01}$-invariants follows.
The  assertion on $G_{00}\ltimes N_{01}$-invariants 
stems form the fact that the $\cF^\bullet_i$'s are already $G_{00}$-invariant.
\end{proof}

\begin{rmk} 
1) It might be true that \eqref{eq:less-restr} already implies that $(\textsl{d}F_{1})_e,\dots,(\textsl{d}F_{k})_e$ are linearly independent
for a $G_{00}$-regular nilpotent $e\in\g_{01}$. Such a phenomenon is observed in several examples. But we unable to either prove or disprove this as yet.

2) Although the algebra $\bbk[\g_{10}{\,\propto\,}\g_{11}]^{N_{01}}$ is polynomial in
Theorem~\ref{thm:main5}(ii), the quotient morphism 
\[
  \pi_{N_{01}}: \g_{10}{\,\propto\,}\g_{11} \to (\g_{10}{\,\propto\,}\g_{11})\md N_{01}\simeq \g_{10}\times \bbk^{k}
\]
is not equidimensional (unless $k=0$, i.e., we are in the situation of Theorem~\ref{thm:no-extra-inv}). Yet it might be true that the quotient morphism by $K=G_{00}\ltimes N_{01}$
\[
  \pi_K: \g_{10}{\,\propto\,}\g_{11} \to (\g_{10}{\,\propto\,}\g_{11})\md (G_{00}\ltimes N_{01})
  \simeq (\g_{10}\md G_{00})\times\bbk^{k} \simeq \bbk^m
\]
is equidimensional. However, our methods for verifying equidimensionality require
an information on the dimension of the linear span of 
$\{(\textsl{d}F_{1})_v,\dots,(\textsl{d}F_{k})_v\}$
for nilpotent elements of $v\in\g_{10}$, which is difficult to infer in general  
(cf. proof of \cite[Theorem\,5.3]{coadj07}).
\end{rmk} 

\noindent
There is a sufficient condition that guarantees that Theorem~\ref{thm:main5}
applies in full strength.

\begin{prop}  \label{prop:dostat-usl}
Let $e\in\g_{10}$ be a $G_{00}$-regular nilpotent element. If $e$ is also $G_{0\star}$-regular
as an element of $\g_{1\star}$, then then all the assumptions of Theorem~\ref{thm:main5}
are satisfied, and hence both $\bbk[\g_{10}{\,\propto\,}\g_{11}]^{N_{01}}$ and
$\bbk[\g_{10}{\,\propto\,}\g_{11}]^{G_{00}\ltimes N_{01}}$ are polynomial algebras.
\end{prop}
\begin{proof}
Let $\{e,h,f\}\subset \g_{\star 0}$ be a normal $\tri$-triple. If $e$ is both $G_{00}$-regular
and $G_{0\star}$-regular, then the same is true for the semisimple element
$e+f\in \g_{10}$  \cite{kr71}. Hence condition~\eqref{eq:less-restr} is satisfied.
It was also explained above that the $G_{0\star}$-regularity of $e$ implies the linear independence of  the differentials for \un{all} basic invariants in $\bbk[\g_{1\star}]^{G_{0\star}}$.
\end{proof}

We  consider below some applications of Theorem~\ref{thm:main5}(ii). 

\begin{ex}   \label{ex:canonical-invar}
Let $\{\vartheta,\vartheta',\mu\}$ be a canonical triple of involutions of $\g$, i.e.,
$\vartheta,\vartheta'$ are of maximal rank and $\mu$ is quasi-maximal.
The corresponding quaternionic decomposition 

\hbox to \textwidth{
\hfil 
{\setlength{\unitlength}{0.02in}
\begin{picture}(35,27)(0,4)
\put(-12,7){$\g=$}
    \put(8,15){$\g_{00}$}    \put(28,15){$\g_{01}$}
    \put(8,3){$\g_{10}$}      \put(28,3){$\g_{11}$}
\qbezier[20](5,9),(22,9),(40,9)              
\qbezier[20](23,-1),(23,11),(23,24)       
\put(19.9,7){$\oplus$}   
\put(43,7){{\color{my_color}$\vartheta$}}
\put(21,-9){{\color{my_color}$\vartheta'$}}
\end{picture}  \hfil
}}
\vskip2.5ex
\noindent
has a number of good properties, see Proposition~\ref{prop:ss-com-inner} and Remark~\ref{rem:nilp-g1}. We wish to describe the algebras of invariants for all  degenerated
isotropy representations.
By Corollary~\ref{cor:sovpad1}, there are at least four coincidence of \CSS; hence the invariants for the respective degenerated isotropy representations are given by 
Theorem~\ref{thm:inv-degen}.  If $\vartheta$ is inner, i.e.,
$\{\vartheta,\vartheta',\mu\}$ is actually a triad, then all little and big \CSS\ are Cartan
subalgebras and the invariants of all six degenerations are described by 
Theorem~\ref{thm:inv-degen}(iii). But, if $\vartheta$ is not inner,
then there are two degenerated isotropy representations, $(G_{00}\ltimes N_{01}: \g_{10}{\,\propto\,}\g_{11})$ and
$(G_{00}\ltimes N_{10}: \g_{01}{\,\propto\,}\g_{11})$, where Theorem~\ref{thm:inv-degen} does not apply. 
We show that the hypotheses of
Theorem~\ref{thm:main5} are satisfied for them, and hence the their
algebras of invariants are also polynomial. 

Assume below that $\vartheta$ is not inner, so that $\vartheta$ and $\mu$ are not 
conjugate. Then $k_0\ne 0$ and  
\[
\dim\g_{00}=\frac{\dim U-k_1}{2}, \ 
\dim\g_{01}=\dim\g_{10}=\dim\g_{00}+k_1,  \ 
\dim\g_{11}=\dim\g_{00}+\rk\g 
\]
(see the proof of Prop.~\ref{prop:ss-com-inner}).
We also have $\rk\g^\vartheta=\rk\g^{\vartheta'}=k_1$ (Proposition~\ref{k0-and-k1}).
As $\vartheta$ is of maximal rank, any \CSS\  $\ce_{1\star} \subset\g_{1\star}$ is a Cartan subalgebra of $\g$.
Recall that 
$\bbk[\g_{1\star}]^{G_{0\star}}\simeq \bbk[\ce_{1\star}]^{\tilde W}$, where
$\tilde W$ is the generalised Weyl group.
In this case, $\tilde W$ coincides with the usual Weyl group of $\g$ with respect to
$\ce_{1\star}$. Hence $\bbk[\g_{1\star}]^{G_{0\star}}\simeq \bbk[\g]^G$.
Let $\ce_{10}$ be a \CSS\ in $\g_{10}$. Then $\dim\ce_{10}=k_1=\rk\g-k_0$.
Hence $\ce_{10}$ cannot be a \CSS\ in 
$\g_{1\star}$. However, $\ce_{10}$ does contain regular semisimple elements of  $\g$ (see the proof of Prop.~\ref{prop:ss-com-inner}), i.e., \eqref{eq:less-restr} is satisfied.

In this situation, the number $k$ occurring in Theorems~\ref{thm:ci} and \ref{thm:main5}
is $k_0$ and basic invariants $F_1,\dots,F_{k_0}$ are precisely the basic invariants 
in $\bbk[\g]^G$ of odd degrees. (Recall that we have defined $k_0$ as the number of 
even exponents of $\g$.) If $e\in\g_{10}$ is $G_{00}$-regular nilpotent, then
$e$ is also regular nilpotent in $\g_{\star 0}=\g^{\vartheta'}$, 
since this involution of $\g_{\star 0}$ is of maximal rank, see \cite{leva}. 
Therefore, the hypothesis of 
Theorem~\ref{thm:main5}(ii) can be restated as
follows: 

{\it Let $\vartheta$ be an (outer) involution of maximal rank and
$e$ a regular  nilpotent element of $\g^{\vartheta}$. Let $F_1,\dots, F_{k_0}\in \bbk[\g]^G$ be the basic 
invariants of odd degrees. Then 
$(\textsl{d}F_{1})_e,\dots,(\textsl{d}F_{k_0})_e$ are linearly independent.}

We verify this for all relevant simple Lie algebras.

a$_1$) $\g=\mathfrak{sl}_{2n+1}$ and $\g^{\vartheta}=\mathfrak{so}_{2n+1}$.
Here $e$ is regular in the whole of $\g$.

a$_2$) $\g=\mathfrak{sl}_{2n}$ and $\g^{\vartheta}=\mathfrak{so}_{2n}$.
The partition of $e$ is $(2n-1,1)$, i.e., $e$ is a subregular nilpotent element of $\g$. 
Let  $f_i$ be a basic invariant  of degree $i$ ($i=2,3,\dots,2n$).
It is known that, for a subregular nilpotent element $e$,  $(\textsl{d}{f}_{2n})_e=0$
(i.e., for the basic invariant of maximal degree), while
the remaining differentials are linearly independent \cite[Lemma\,5.1]{ri2}.
This remaining set contains all basic invariants of odd degree.

a$_3$) $\g=\mathfrak{so}_{4n+2}$ and $\g^{\vartheta}=\mathfrak{so}_{2n+1}\times
\mathfrak{so}_{2n+1}$. The partition of $e$ is $(2n+1,2n+1)$ and $k_0=1$. 
The only basic invariant of odd degree is the pfaffian $\mathsf{P}$. One  verifies that, 
for a nilpotent element $v$, $(\textsl{d}\mathsf{P})_v=0$ if and only if the partition of $v$ has 
at least three nonzero parts \cite[Lemma\,4.1.1]{ri2}.

a$_4$)  $\g=\GR{E}{6}$ and $\g^{\vartheta}=\mathfrak{sp}_8$. 
Here $k_0=2$, $\deg F_1=5$ and $\deg F_2=9$.
As $e$ appears to be subregular in $\g$, we can again use  \cite[Lemma\,5.1]{ri2}.

{\sl Thus, in case of  canonical decompositions, we can describe the algebras of invariants
for all six degenerated isotropy representations.}
\end{ex}

\begin{ex}   \label{ex:esche}
Let $\g$ be of type $\GR{E}{6}$. According to \cite[Table\,1]{kollross},
there are  three commuting involutions with the following fixed-point subalgebras:

\quad $\sigma_1, \sigma_2$:  \   $\GR{F}{4}$; \quad
$\sigma_3$:  \ $\GR{D}{5}\oplus\te_1$.

\noindent As $\{\sigma_1,\sigma_2\}$ is supposed to be a dyad, this can also be 
verified using the restricted
root system for $\GR{E}{6}/\GR{F}{4}$, which is reduced and of type $\GR{A}{2}$
(see Section~\ref{sect2}).
Here $\g_{00}=\mathfrak{so}_9$ and the $\g_{00}$-modules $\g_{ij}$ are:

\begin{center}
$\g_{01}\simeq \g_{10}\simeq \mathsf R_{\varpi_4}$, \ $\g_{11}\simeq \mathsf R_{\varpi_1}+
\mathsf R_0$,
\end{center}

\noindent 
where $\varpi_i$'s are fundamental weights of $\mathfrak{so}_9$ and $\mathsf R_\lb$ stands
for  the simple module with highest weight $\lb$.
Hence $\dim\g_{01}=\dim\g_{10}=16$ and $\dim\g_{11}=10$.
Here $\dim\ce_{10}=\dim\ce_{01}=1$ and $\dim\ce_{11}=2$. By Theorem~\ref{thm:sovpad2},
$\ce_{11}$ is also a \CSS\ in $\g_{1\star}$ or $\g_{\star 1}$; and these are the only 
coincidences of \CSS\ here. In view of these coincidence and 
Theorem~\ref{thm:inv-degen}(ii),
the algebra of invariants for two isomorphic degenerated isotropy representations,
$(G_{00}\ltimes N_{01}: \g_{11}{\,\propto\,}\g_{10})$ and 
$(G_{00}\ltimes N_{10}: \g_{11}{\,\propto\,}\g_{01})$,
is readily described:
\[
\bbk[(\mathsf R_{\varpi_1}{+}
\mathsf R_0){\,\propto\,}\mathsf R_{\varpi_4}]^{\text{Spin}_9\ltimes \exp(\mathsf R_{\varpi_4}^a)}\simeq \bbk[\mathsf R_{\varpi_1}+
\mathsf R_0]^{\text{Spin}_9} .
\]
Therefore,  this algebra has the basic invariants of degree 
$1$ and $2$. The representation 
\beq   \label{spin9-1}
      \bigl(\text{Spin}_9\ltimes \exp(\mathsf R_{\varpi_4}^a): (\mathsf R_{\varpi_1}{+}
\mathsf R_0){\,\propto\,}\mathsf R_{\varpi_4}\bigr)
\eeq
is a degeneration of the isotropy representation of
the symmetric space $\GR{E}{6}/\GR{F}{4}$, i.e., of 
$(\GR{F}{4}:\mathsf R)$, where $\dim\mathsf R=26$ and the degrees of basic invariants are
$2$ and $3$. The other degeneration of the same isotropy representation is
\beq   \label{spin9-2}
      \bigl(\text{Spin}_9\ltimes \exp(\mathsf R_{\varpi_4}^a): \mathsf R_{\varpi_4}{\,\propto\,}(\mathsf R_{\varpi_1}{+}
\mathsf R_0)\bigr) .
\eeq
Here we can exploit the fact that any $G_{00}$-regular nilpotent element in 
$\g_{10}\simeq \mathsf R_{\varpi_4}$ appears to be $G_{0\star}$-regular  in 
$\g_{1\star}\simeq \mathsf R$. Indeed, using the Satake diagram for
the symmetric pair $(\GR{F}{4}, \mathfrak{so}_9)$, one verifies that the 
$\text{Spin}_9$-regular nilpotent element of $\g_{10}$ belong to the nilpotent $\GR{F}{4}$-orbit $\co$,
denoted $\GRt{A}{2}$. On the other hand, using the Satake diagram for
the symmetric pair $(\GR{E}{6}, \GR{F}{4})$, one verifies that the 
$\GR{F}{4}$-regular nilpotent element of $\g_{1\star}$ belongs to the nilpotent 
$\GR{E}{6}$-orbit $\tilde\co$, denoted $2\GR{A}{2}$. It is not hard to verify using the respective weighted Dynkin diagrams that $\tilde\co\cap \f_4=\co$.

Since $\dim\ce_{1\star}-\dim\ce_{10}=1$, $X=\ov{G_{0\star}{\cdot}\g_{10}}$ \ is a hypersurface in
$\g_{1\star}$, and this hypersurface must be the zero set of the basic invariant of degree $3$ in $\bbk[\mathsf R]^{\GR{F}{4}}$.
Finally, applying Theorem~\ref{thm:main5} and Prop.~\ref{prop:dostat-usl} shows that 
$\bbk[\mathsf R_{\varpi_4}{\,\propto\,}(\mathsf R_{\varpi_1}{+}\mathsf R_0)]^{\text{Spin}_9\ltimes \exp(\mathsf R_{\varpi_4}^a)}$ is a polynomial algebra with basic invariants of degree $2$ and $3$.

Thus, the dual representations \eqref{spin9-1} and \eqref{spin9-2} have polynomial algebras
of invariants, but the degrees of basic invariants are different.

Set $V=(\mathsf R_{\varpi_1}{+}\mathsf R_0){\,\propto\,}\mathsf R_{\varpi_4}$, 
$V^*=\mathsf R_{\varpi_4}{\,\propto\,}(\mathsf R_{\varpi_1}{+}\mathsf R_0)$, and 
$K=\text{Spin}_9\ltimes \exp(\mathsf R_{\varpi_4}^a)$.
Since $\dim V\md K=\dim V^*\md K=2$ and $K$ has no rational characters, one easily derives that both morphisms
$\pi_V:  V\to V\md K\simeq \mathbb A^2$ and 
$\pi_{V^*}: V^* \to V^*\md K\simeq \mathbb A^2$ are equidimensional; hence
$\bbk[V]$ is a free $\bbk[V]^K$-module, and likewise for $V^*$.
\end{ex}

\section{Problems and observations}  
\label{sect6}\nopagebreak
\subsection{}
In \cite{coadj07}, I suggested that  the coadjoint representation
of a $\BZ_2$-contraction of $\g$  always has a polynomial algebra of invariants.
However, recent work of O.\,Yakimova \cite{zhena}
demonstrates that this is not always the case. 
Therefore, there ought to be quaternionic decompositions of a {\sl simple\/} Lie algebra $\g$ such that some degenerated isotropy
representations $(G_{00}\ltimes N_\beta : \g_\ap{\,\propto\,}\g_\gamma)$ do not have
a polynomial algebra of invariants.  On the other hand, results of Sections~\ref{sect4} and
\ref{sect5} show that in many cases the algebra 
$\bbk[\g_\ap{\,\propto\,}\g_\gamma]^{G_{00}\ltimes N_\beta}$ 
is polynomial. This raises the following

\begin{quest}   \label{vopros1}
What are more precise necessary and/or sufficient conditions for 
$\bbk[\g_\ap{\,\propto\,}\g_\gamma]^{G_{00}\ltimes N_\beta}$ to be polynomial?
\end{quest}

In Section~\ref{sect4}, we proved that if a \CSS\ $\ce_\ap\subset\g_\ap$ is also 
a \CSS\ in $\g_\ap\oplus\g_\gamma$, then $\bbk[\g_\ap{\,\propto\,}\g_\gamma]$
is a free $\bbk[\g_\ap{\,\propto\,}\g_\gamma]^{G_{00}\ltimes N_\beta}$-module, i.e.,
the quotient morphism 
\[
 \pi: \g_\ap{\,\propto\,}\g_\gamma\to (\g_\ap{\,\propto\,}\g_\gamma)\md G_{00}\ltimes N_\beta
 \simeq \mathbb A^m
\]
is equidimensional. This raises our next

\begin{quest}
What are more precise necessary and/or sufficient conditions for 
$\bbk[\g_\ap{\,\propto\,}\g_\gamma]$ to be a free
$\bbk[\g_\ap{\,\propto\,}\g_\gamma]^{G_{00}\ltimes N_\beta}$-module?
\end{quest}

Of course, both Questions are prompted by the fact that the respective properties are 
always satisfied for the initial (big) isotropy representations, see Section~\ref{sect1}.

\subsection{}  \label{subs:dif-op}
Let $Q$ be an algebraic group with $\Lie(Q)=\q$. Suppose that
$\sigma$ is an involution of $Q$  and $Q_0$ is
the identity component of $Q^\sigma$. Let $H$ be any subgroup between
$Q_0$ and $Q^\sigma$ such that $Q/H$ is connected (this is always the case if $Q$ is connected).  
Then the algebra, $\mathcal D(Q/H)$, 
of left invariant differential operators on the homogeneous space $Q/H$ is commutative~\cite{smoke},\cite{duflo}.
Let $\q=\q_0\oplus\q_1$ be the corresponding $\BZ_2$-grading. The following is
proved in \cite[Lemma\,4.2]{gonz-helg}:
\vskip.7ex
\centerline{
{\it If the algebra $\eus S(\q_1)^{H}=\bbk[\q_1^*]^{H}$ is polynomial, then so is\/ 
$\mathcal D(Q/H)$.}}
\noindent
In fact, there is a canonical linear bijection between the two algebras that transforms the free 
generators of the former to the free generators of the latter.

Therefore, our results on quaternionic decompositions provide a description
of invariant differential operators for certain degenerations of symmetric spaces of $G$.
Recall that the group $K_{01}=G_{00}\ltimes N_{01}$ can be regarded as
the identity component of a symmetric subgroup in two different ways:

a)  $K_{01}$ is the identity component of
$G\langle\sigma_2\rangle^{\sigma_1}$, and the corresponding isotropy representation is
$(K_{01}:  \g_{10}{\,\propto\,}\g_{11})$;
 
b) $K_{01}$ is the identity component of $G\langle\sigma_3\rangle^{\sigma_1}$,
and the corresponding isotropy representation is
$(K_{01}:  \g_{11}{\,\propto\,}\g_{10})$.

Taking into account that $(\g_{10}{\,\propto\,}\g_{11})^*= \g_{11}{\,\propto\,}\g_{10}$,
we obtain the following:

\begin{utv} If\/  $\bbk[\g_{10}{\,\propto\,}\g_{11}]^{K_{01}}$ is a polynomial algebra, then so 
is\/ $\mathcal D(G\langle\sigma_3\rangle/K_{01})$;
\\
if\/  $\bbk[\g_{11}{\,\propto\,}\g_{10}]^{K_{01}}$ is a polynomial algebra, then so 
is\/ $\mathcal D(G\langle\sigma_2\rangle/K_{01})$.
\end{utv}

It might be interesting to explore a  relationship between the joint eigenfunctions for  $\mathcal D(G/G_{0\star})$ and $\mathcal D(G\langle\sigma_3\rangle/K_{01})$.
Also, if $\ce_{10}$ is a \CSS\ in $\g_{01}\oplus\g_{11}$, then
$\bbk[\g_{10}{\,\propto\,}\g_{11}]^{K_{01}}\simeq \bbk[\g_{10}]^{G_{00}}$, which seems
to suggest that
there might be a direct relation between the polynomial algebras of differential operators
 $\mathcal D(G\langle\sigma_3\rangle/K_{01})$ and $\mathcal D(G_{\star 0}/G_{00})$.

\subsection{} 
In \cite{gonz-helg}, Gonzalez and Helgason determined the algebra of invariant differential operators for an interesting non-reductive symmetric pair $(Q,Q^\sigma)$. 
In their exposition, the (non-connected) group $Q$ is  the semi-direct product of
the orthogonal group $\eus O_n=\eus O(V)$ and the standard $\eus O(V)$-module $V$, i.e., using our notation, $\q=\sov\ltimes V^a$ and 
$Q=\eus O(V)\ltimes \exp(V^a)$. Then they consider an involution $\sigma$ of $Q$ 
such that $Q^\sigma= (\eus O(V_1)\ltimes \exp(V_1^a))\times \eus O(V_2)$, where 
$V=V_1\oplus V_2$ is an orthogonal direct sum. Set $p=\dim V_1$, so that 
$n-p=\dim V_2$. 

Our observation is that $Q/Q^\sigma$ is  a degeneration of the symmetric space
$\eus O_{n+1}/\eus O_{p+1}\times \eus O_{n-p}$, and this degeneration is related to 
a quaternionic decomposition of $\g=\mathfrak{so}_{n+1}$. 
To get this, we assume that $\mathfrak{so}_{n+1}=
\mathfrak{so}(\tilde V)$ consists of usual skew-symmetric matrices of order $n+1$ with 
respect to a certain basis of $(n+1)$-dimensional
space $\tilde V$ and consider the  involutions 
$\sigma_i$ ($i=1,2,3$) \ defined by the  diagonal matrices $s_1,s_2,s_3\in \eus O(\tilde V)$ :
\vskip.7ex
$s_1=
\text{diag}(\underbrace{1,\dots,1}_{p+1},\underbrace{-1,\dots,-1}_{n-p})$,
\quad 
$s_2=\text{diag}(\underbrace{1,\dots,1}_{p},\underbrace{-1,\dots,-1}_{n-p+1})$,
\quad 
$s_3=s_1s_2$.

\vskip.8ex
Let $V_1$ (resp. $V_2$) be the linear span of the first $p$ (resp. last $n-p$) basis vectors.
Then $\g_{00}=\mathfrak{so}(V_1)\times \mathfrak{so}(V_2)$ and the whole quaternionic
decomposition is:
\begin{center}
\setlength{\unitlength}{0.02in}
\begin{picture}(90,30)(-20,0)
\put(-44,7){$\g=$}
    \put(-25,15){$\mathfrak{so}(V_1){\times} \mathfrak{so}(V_2)$}    \put(45,15){$V_1$}
    \put(0,0){$V_2$}      \put(38,0){$V_1\otimes V_2$}
\qbezier[50](-24,9),(23,9),(70,9)            
\qbezier[20](33,-3),(33,11),(33,24)      
\put(29.9,7){$\oplus$}   
\put(31,-12){{\color{my_color}$\sigma_2$}}
\put(80,7){{\color{my_color}$\sigma_1$}}
\end{picture}  
\end{center}
\vskip3.5ex
(i.e., $\g_{10}\simeq V_2$ as $\g_{00}$-module, etc.)
Here $\tilde V=V_1\oplus V_2\oplus \bbk=V\oplus \bbk$, \ 
$\g^{\sigma_1}=\mathfrak{so}(V_1\oplus \bbk)\times\mathfrak{so}(V_2)$, 
$\g^{\sigma_2}=\mathfrak{so}(V_2\oplus \bbk)\times \mathfrak{so}(V_1)$,
and $\g^{\sigma_3}=\mathfrak{so}(V)$.
In our notation, the Lie algebra $\Lie(Q)=\q$ considered in \cite{gonz-helg} is nothing but
$\g\langle\sigma_3\rangle$;
and  Theorem\,3.1 in \cite{gonz-helg} essentially means that, 
for $H:=(\eus O(V_1)\times \eus O(V_2))\ltimes \exp(V_1^a)$,
the algebra 
$\bbk[V_2{\,\propto\,}(V_1\otimes V_2)]^H$ is polynomial, of 
Krull dimension  $\min\{p+1, n-p\}$. 
In our setting, the dimension formula is explained by Corollary~\ref{cor:4.12} and 
the fact that $\rk (\eus{SO}_{n+1}/\eus{SO}_{p+1}\times \eus{SO}_{n-p})=\min\{p+1,n-p\}$.
The identity component of 
$H$ is 
\[
K_{01}=(\eus{SO}(V_1)\times \eus{SO}(V_2))\ltimes \exp(V_1^a) ,
\] 
and 
one can prove that, for this degenerated isotropy representation, we have a materialisation 
of the phenomenon of Remark~\ref{rmk:bullet}, i.e., 
$\gr^\bullet (\bbk[\g_{1\star}]^{G_{0\star}})=\bbk[V_2{\,\propto\,}(V_1\otimes V_2)]^{K_{01}}$.
Note that $G_{0\star}=\eus{SO}(V_1\oplus\bbk)\times \eus{SO}(V_2)$ and the basic invariants of $\bbk[\g_{1\star}]^{G_{0\star}}$ have degrees 
\[ 
\begin{cases}
2,4,\dots ,2\min\{p+1,n-p\}  & \ \text{ \ if \ } p+1\ne n-p \ , \\
2,4,\dots, 2p, p+1   & \ \text{ \ if \ } p+1=n-p \ .    
\end{cases}
\]
Therefore, the same is true for $\bbk[V_2{\,\propto\,}(V_1\otimes V_2)]^{K_{01}}$.
Since $\dim \ce_{10}=1$ and $\dim\ce_{1\star}=\min\{p+1, n-p\}$, the 
coincidence of \CSS\ occurs only if $n-p=1$. 

The advantage of using quaternionic decompositions
is that we immediately get other interesting possibilities. For the dual $K_{01}$-module
$(V_1\otimes V_2){\,\propto\,}V_2$, we may compare the Cartan subspaces 
$\ce_{11}\subset V_1\otimes V_2$ and $\ce_{1\star}$. 

\textbullet \ \ Since $\dim\ce_{11}=\min\{p,n-p\}$ and $\dim\ce_{1\star}=\min\{p+1, n-p\}$,
we have a coincidence of \CSS\ if and only if $n-p\le p$, and then 
$\bbk[(V_1\otimes V_2){\,\propto\,}V_2]^{K_{01}}\simeq
\bbk[V_1\otimes V_2]^{\eus{SO}(V_1)\times \eus{SO}(V_2)}$. The latter is a polynomial algebra whose degrees of basic invariants are 

$2,4,\dots,2(n-p)$, \ if $n-p<p$; \quad 
$2,4,\dots,2p-2, p$, \ if $n-p=p$.

\textbullet \ \ For $p+1=n-p$, one can notice that a $G_{00}$-regular nilpotent element
in $\g_{11}\simeq V_1\otimes V_2$ is actually regular in the whole of $\g$.
Hence, Proposition~\ref{prop:dostat-usl} applies to 
$K_{01}: (V_1\otimes V_2){\,\propto\,}V_2$ and hence
$\bbk[(V_1\otimes V_2){\,\propto\,}V_2]^{K_{01}}$ is polynomial.

\textbullet \ \ Even for $p+1< n-p$, we can prove using the `contraction procedure' for 
$G_{0\star}$-invariants that $\bbk[(V_1\otimes V_2){\,\propto\,}V_2]^{K_{01}}$ is a 
polynomial algebra. The details will appear elsewhere.

As explained in Subsection~\ref{subs:dif-op}, 
this implies that 
$\mathcal D(G\langle\sigma_2\rangle/K_{01})$ is a polynomial algebra, where
$G=\eus{SO}(V_1\oplus V_2\oplus \bbk)$ and 
$G\langle\sigma_2\rangle=\bigl(\eus{SO}(V_1)\times \eus{SO}(V_2\oplus \bbk)\bigr)\ltimes (V_1\otimes (V_2\oplus\bbk))$.

\vskip1ex
\noindent
{\small  {\bf Acknowledgements.} Part of this work was done while I was able to use
rich facilities of the Max-Planck Institut f\"ur Mathematik (Bonn).}

\end{document}